\documentclass{amsart}
\usepackage{amssymb}
\usepackage{booktabs}
\usepackage{mathtools}
\usepackage{minibox}
\usepackage{psfrag}
\usepackage{xr-hyper}
\usepackage[colorlinks]{hyperref}


\thanks{A.A. was partially supported by  CNPq,  CAPES e PRONEX FAPERJ. D.S. was partially supported by CNPq 306622/2019-0, CNPq 307617/2016-5, CNPq Universal 430351/2018-6 and FAPESP Projeto Tem\'atico 2017/06463-3.}

\title{Transfer operators and  atomic decomposition }

\author[A. Arbieto and D. Smania]{Alexander Arbieto and Daniel Smania}
\address{Instituto de Matem\'atica, Universidade Federal do Rio de Janeiro, P. O. Box 68530, 21945-970 Rio de Janeiro, Brazil.}
\address{Departamento de Matem\'atica, Instituto de Ci\^encias Matem\'aticas e de Computa\c{c}\~ao-Universidade de S\~ao Paulo (ICMC/USP), Caixa Postal 668, S\~ao Carlos-SP, Brazil.}
\email{arbieto@im.ufrj.br}
\email{smania@icmc.usp.br}
\urladdr{\url{https://sites.icmc.usp.br/smania/}}

\newtheorem{theorem}{Theorem}[section]
\newtheorem{corollary}{Corollary}[section]

\newtheorem{lemma}[theorem]{Lemma}
\newtheorem{proposition}[theorem]{Proposition}

\theoremstyle{definition}
\newtheorem{definition}[theorem]{Definition}

\usepackage{etoolbox}
\usepackage{constants} 

\newconstantfamily{name}{
symbol=C,
format=\secdot,
}

\newconstantfamily{namec}{
symbol=\lambda,
format=\secdot,
}

\newconstantfamily{A}{
symbol=A,
format=\arabic,
}

\newcommand{\secdot}[1]{\arabic{#1}}
\newconstantfamily{c}{
symbol=\lambda,
format=\secdot,
reset={},
}

\renewconstantfamily{normal}{
symbol=C,
format=\secdot,
}

\newconstantfamily{e}{
symbol=\theta,
format=\arabic,
}

\makeatletter
\newcommand{\Cll}[2][normal]{%
\global\@namedef{cstt@#1@#2}{}%
\@ifundefined{cstt@name@#2} {\@ifundefined{cstt@namec@#2}{ { \Cl[#1]{#2}}}{ {\lambda_{_{#2}}}}}{{C_{_{#2}}}} %
}
\newcommand{\Crr}[1]{\@ifundefined{cstt@name@#1}{   \@ifundefined{cstt@namec@#1}{\Cr{#1}}{\lambda_{_{#1}}}}{C_{_{#1}}}}
\makeatother

\newcounter{change}

\usepackage[colorinlistoftodos, textwidth=4cm]{todonotes}

\makeatletter
\providecommand\@dotsep{5}
\renewcommand{\listoftodos}[1][\@todonotes@todolistname]{%
  \@starttoc{tdo}{#1}}
\makeatother

\hypersetup{ colorlinks=true, pdftitle={Transfer operators and  atomic decomposition}, pdfsubject={ transfer operator, atomic decomposition, Besov space,  Ruelle,  Perron-Frobenious,  quasi-compact, Lasota-Yorke, decay of correlations, expanding map}, pdfauthor={Alexander Arbieto and Daniel Smania},pdfkeywords={ transfer operator, atomic decomposition, Besov space,  Ruelle,  Perron-Frobenious,  quasi-compact, Lasota-Yorke, decay of correlations, expanding map}}

\begin{document}

\begin{abstract} 
We use  the method of atomic decomposition and a new family of Banach spaces to  study the action of  transfer operators associated with  piecewise-defined  maps. It turns out that these transfer operators  are quasi-compact even when  the associated potential, the dynamics and  the underlying phase  space have very low regularity. 

 In particular,  it is often possible to obtain exponential decay of correlations, the Central Limit Theorem and almost sure invariance principle for fairly general observables, including unbounded ones. 
\end{abstract}

\maketitle

\setcounter{tocdepth}{2}
\tableofcontents

\vspace{1cm}
\centerline{ \bf I. INTRODUCTION.}
\addcontentsline{toc}{chapter}{\bf I. INTRODUCTION.}
\vspace{1cm}

Transfer operators are an almost unavoidable tool to study the ergodic theory of (piecewise) smooth dynamical systems. In the context of expanding maps, we usually have a reference measure, that can be for instance  either  the volume form  on the manifold where the dynamics $F$ takes place, or in more general settings some "eigen-measure"  $m$  for the dual operator (to find such eigen-measure quite often it is not a trivial matter).  The transfer  operator describes how finite measures that  are absolutely  continuous with respect to $m$ are transported by the dynamics. That is,  if  $\mu = \rho m$, with $\rho \in L^1(m)$  then $\Phi \rho$ is the density of the push-forward $F^\star\mu$ with respect to $m$. \\

\noindent \textit{ We consider  new  \textbf{Besov spaces} on  phase spaces with  very mild   structure, a finite \textbf{measure space with a good grid}, and 
we   estimate the (essential) spectral radius of  \textbf{transfer operators of  piecewise-defined maps} acting  on them.   These spaces  often  coincide with classical Besov spaces in more familiar settings. On the other hand, the assumptions on the regularity of  both the map and the phase space are minimal, allowing us to apply the results to new and classical situations alike. We use the  \textbf{atomic decompositon}  of these Besov  spaces to study the action of the transfer operator. } \\

\section{Transfer operators and dynamics} We know that even for very regular expanding maps, typical $L^1$ observables {\it do not have} good statistical properties as exponential decay  of correlations and central limit theorem. Indeed this is  related to the bad spectral behaviour  of the action on $\Phi$ on $L^1(m)$. The transfer operator acts as a bounded operator on $L^1(m)$, but  its spectrum there is {\it  the whole closed unit disc.} This  obviously  have dreadful consequences for the decay of correlations of typical integrable observables.\\

The most well-behaved linear operators are linear transformations on finite-dimensional normed spaces. Its spectrum is just a finite number of eigenvalues with finite-dimensional eigenspaces.  The next best thing would be compact operators, for which the spectrum are just a countable number of eigenvalues, possibly accumulating at zero.  Unfortunately, the transfer operator is  very rarely a compact operator  even in very regular situations. A far more  successful approach   to obtain  good statistical properties of observables in some Banach space of functions $B$, often called {\it functional operator approach},  is to show the {\it quasi-compactness}  of the action of $\Phi$  on $B$, that is, the spectrum near the circle of espectral  radius is as of a compact operator, consisting on  isolated eigenvalues and finite-dimensional eigenspaces, and the "weird  stuff", the so-called essential spectrum, safely away from it, inside a disc of strictly smaller radius. \\

One must note that the quasi-compactness of $\Phi$ is not the only difficulty  in the functional operator approach, however, it is fair to say that finding a proper Banach space of functions and  proving the quasi-compactness of $\Phi$ there it is one of the most challenging steps.  There are well-known methods for using the quasi-compactness property to study the ergodic behaviour of $F$. We list a (purposely vague) description of some of them below. Section \ref{lyorke} and Section \ref{lyorke2}, give precise statements of some consequences of the quasi-compactness of transfer operators action on Besov spaces on measure spaces with a good grid. \\

\noindent {\it Existence of absolutely continuous invariant probabilities.} To this end, one needs to show that $1$ is an eigenvalue of $\Phi$ and its eigenspace contains a non-negative function $\rho$.  Then the measure $\rho \ dm$ is a finite  invariant  measure. One can also estimate the number of absolutely continuous ergodic measures by the dimension of this eigenspace if $\Phi$ also satisfies  the quite handy  Lasota-Yorke inequality  for the pair of Banach spaces $(B,L^1(m))$.\\

\noindent {\it Exponential decay of correlations.}  One needs to show that $1$ is an isolated simple eigenvalue and that the rest of the spectrum is contained in a ball centered at zero and with radius strictly  smaller than one.  Exponential decay of correlation follows for all observables in $B$.\\

\noindent {\it Central Limit Theorem.} To show the Central Limit Theorem for a real-valued observable $\phi$ we first show that $\psi \mapsto  M_t(\psi)=e^{i t \phi}\psi$  is a bounded operator (a {\it multiplier}) on $B$, for every $t$ small. Then  we  consider perturbations of  the transfer operator $\Phi_t= \Phi\circ M_t$. Often there is an analytic continuation of the  leading eigenvalue for every small $t$. This is closely related to the characteristic function of the observable  $\phi$, and a careful analysis gives the Central Limit Theorem for $\phi$. Note that $\Phi_t$ is also a transfer operator, but with a complex-valued potential.\\

\noindent {\it Analyticity of topological pressure}. If $I$ is a compact metric space and $F\colon I \rightarrow I$ is continuous  then the spectral radius of the operator $\Phi_g$ with potential $g$ is exactly $e^{P_{top}(g)}$. If there is just a single  element of the spectrum with maximal modulus, which is a simple eigenvalue,  this  eigenvalue often varies analytically under perturbation of $g$, so  we get the real analyticity of $P_{top}(g)$ with respect to $g$.  \\

\section{Looking for Banach spaces} 

There is a long history of looking for Banach spaces of observables with good statistical properties. We give below  a list of  dynamics, potentials  and corresponding function spaces where the quasi-compactness of the transfer operator was attained.  \\

The transfer operator (for the potential $\log  f'$) appeared in  1956 in  Rechard \cite{rechard} as a tool to find invariant measures of one-dimensional many-to-one transformations. But  its  impact  certainly had a little help of the popular book by Ulam \cite{ulam1}, where he asks if one could show a  result similar to Perron-Frobenious Theorem for positive matrices.  \\

Results on the spectral theory  of the transfer operator (as named  by Ruelle  but often called Ruelle-Perron-Frobenious operator) started with the seminal work on rigorous statistical mechanics  by  Ruelle \cite{ruelle1d}, who studied the one-sided shift and the action on H\"older function of the transfer operator associated with potentials in the same class, and, in particular, got a result analogous to the   Perron-Frobenious Theorem in this setting. \\

The construction of Markov partitions for hyperbolic maps by Sinai \cite{markovp} allowed to study transfer operators for expanding maps on manifolds \cite{tf}  and compact sets with the same Banach spaces of functions, since they have Markov partitions that semiconjugate them with subshifts of finite type.  See also Ruelle \cite{rjulia}, Parry \cite{parry},  Walters \cite{waltersgm}, Bowen \cite{qs}, Bowen and Series \cite{series}.  See  Bowen \cite{bowen}, Parry and Pollicott \cite{pp}, Przytycki  and Urba\'{n}ski \cite{pu}, Zinsmeister \cite{zin}, as well  Craizer \cite{craizer}   for    superb expositions on this setting, with different emphasis in its applications.  \\

Lasota and Yorke \cite{ly} gave the next step. They considered piecewise $C^2$ expanding maps on the interval, motivated by a quite concrete problem involving the shape of  well drilling bits. Of course, any space of continuous functions is not invariant by the transfer operator anymore. Moreover, the Markov partition approach is no longer easily adaptable here once one needs subshifts that are not of finite type. They proved the action of the transfer operator on the space $BV$ of bounded variation function satisfies what is now called the Lasota-Yorke inequality, which  in particular implies the quasicompactness of the action of the transfer operator. With the exception of early results by  Gel$'$fond \cite{gelfond} and  Parry \cite{parry2} on $\beta$-transformations, and Lasota \cite{lasota}, this was the first time  one could obtain  deep ergodic results for  non-markovian maps.  Keller and Hofbauer  \cite{hk1}\cite{hk2} pushed these results for bounded variation potentials, particularly the quasi-compactness of the transfer operator and its consequences.  Baladi \cite{baladi1} and Broise \cite{broise} are good introductions for these results.  \\ 

 Lasota-Yorke  inequality and quasi-compactness became  favorite tools to study transfer operators.  Keller \cite{kt} studied piecewise complex-analytic expanding maps on the plane. The space  $BV$ was used in higher dimensions to study piecewise expanding maps by  G\'ora  and Boyarsky \cite{MR1029902}.  See also  Adl-Zarabi \cite{ad}. Cowieson \cite{cowieson}\cite{cowieson1} proved the quasi-compactness of the transfer operator for "generic" piecewise $C^k$ expanding maps.  Indeed  in dimension larger than one, the discontinuities  of the dynamics became an even more serious liability. If you pick a piecewise monotone map on the interval  whose branches are defined in intervals, its $n$th iteration has  monotone  branches with the very same property. However, if we iterate a map that is a piecewise expanding map whose domains of the branches are very nice (squares, for instance) then its $n$th iteration may be a piecewise expanding map with branches defined in domains with  increasingly more complex geometry and  the associated  partition may have increasingly complex topology. Consequently,  nearly all these  results depend either  on a priori estimates or hold only for generic maps.  The only exceptions are the results by Buzzi \cite{MR1764923} and   Tsujii \cite{tbv} on  the quasi-compactness of the transfer operator for general piecewise real analytics maps defined in branches with domains whose boundary are piecewise analytic curves. Tsujii used the BV space, while Buzzi used the space introduced by Keller and  Saussol result described below.  There are also results for piecewise affine maps in the plane by Buzzi \cite{affineb} (see also Buzzi \cite{buzzi2}) and   for arbitrary dimension by Tsujji  \cite{tad}. We note that there are examples by Tsujii \cite{tsujiic} and Buzzi \cite{buzzic}  of $C^r$-piecewise expanding maps on $\mathbb{R}^n$ without an absolutely continuous invariant probability. \\

In the late 70's strange attractors attracted the interest of the mathematical community. In particular, Lorenz's attractor poses new problems to ergodic  theory of expanding maps since one can reduce many problems on the dynamics of the Lorenz's flow to the study of a one-dimensional expanding map. However, this map is non-markovian, and it has singularities on which  the derivative blow-up, so the previous function spaces did not work anymore. \\

Keller \cite{keller}  introduced a new space, the spaces of generalized $p$-bounded variation function, that allows him  to get the quasi-compactness of the transfer operator with $p$-bounded variation potential $\log F'$ for one-dimensional maps, including Lorenz maps and piecewise  $C^{1+\alpha}$-maps. This sparked an intense interest in using the same space  to higher-dimensional setting, especially given the difficulty of dealing with BV space in this setting.  Saussol \cite{saussol} result for piecewise H\"older potentials in higher dimensions, that depends on an a priori estimate, was applied by Buzzi \cite{MR1764923} in his result on piecewise real analytics maps in the plane. \\

Note that  generalized $p$-bounded variation function spaces seems to be {\it ad hoc} spaces. Moreover, this  space is also in $L^\infty$, that it is a constraint given that unbounded observables may be handy sometimes.  One may ask if we can get a larger and more familiar space to work with. Indeed Thomine \cite{thomine} obtained a result for Sobolev spaces $H^s_p$, with $0< s < 1/p$ in the  case of $C^r$ piecewise expanding multimodal maps on manifolds (as usual in this setting, the map needs to satisfy an a priori estimate). \\

There are also recent  results for one-dimensional expanding maps by Butterley \cite{butter} and  Liverani \cite{liverani} using some spaces  of functions. The Liverani's  space, in particular, is related to methods to study the transfer operator of hyperbolic maps acting on certain anisotropic Banach spaces. \\

Nakano and Sakamoto \cite{ns} recently obtained the quasi-compactness of the transfer operator  for smooth expanding maps on manifolds without  discontinuities acting on Besov spaces. \\

See also Baladi and Holschneider \cite{bh} for an earlier application of wavelets and multiresolution analysis  in the study of transfer operators of smooth expanding maps on manifolds.\\

If we move away from the functional analytic approach, Eslami \cite{eslami} studied  the decay of correlations for expanding maps on metric spaces. The class of observables under consideration is indeed a cone (rather than a linear subspace) of functions, inspired by the standard pairs  developed by Dolgopyat and Chernov (see for instance   \cite{dolgo1} \cite{dolgo2}). \\

We finish this historic account saying that the development  of the functional analytic approach  for hyperbolic maps (for instance, Anosov diffeomorphisms) has been very intense in the last years, with many exchanges of ideas with results on expanding maps. A fair description of these new developments is  beyond the scope of this work. We refer the reader  to the works of  Blank, Keller and  Liverani \cite{bkl}, Gou\"{e}zel and Liverani \cite{gl},  Baladi and Tsujii \cite{bt}, as well the survey  and the recent book by Baladi \cite{bsur} \cite{baladi2} and the references therein for  more information.

\section{Who needs  yet another  Banach space?}

We offer an {\bf "one-fits-all" approach.}  The Besov spaces on measure spaces with good grids considered here include many of the Banach spaces of functions considered in the literature on transfer operators. In particular Keller's spaces of generalised $p$-bounded variation and Sobolev spaces. Moreover one can cover most dynamics already considered before, as Lorenz 1-dimensional maps and  piecewise $C^{1+\alpha}$ expanding maps,  giving  new statistical results for a {\bf wider class of observables}, including unbounded ones.  We can consider Besov spaces (in particular Sobolev spaces) in many settings \cite{smania-homo}, in particular, homogeneous spaces (a quasi-metric space with a doubling measure), as  for instance  symbolic spaces and hyperbolic Julia sets with an appropriated  reference measure. But it also allows us to {\bf deal with new situations}, as maps with potentials in Besov spaces.  In the companion paper \cite{smania-examples}, we give a long list of applications. Finally, this is a  {\bf  elementary approach.}  Besov spaces on measure spaces with good grids have a fairly elementary definition \cite{smania-besov} and it demands simple methods. In particular, the atomic decomposition by  atoms with discontinuities  is embraced from the very beginning, so  we do not need to deal with mollifiers, which makes the proofs more transparent and straightforward. \\ \\


\vspace{1cm}
\centerline{ \bf II. PRELIMINARIES.}
\addcontentsline{toc}{chapter}{\bf II. PRELIMINARIES.}
\vspace{1cm}

\section{Measure space and good grids}\label{partition}   \label{goodgrids}
 Let $I$  be a measure space with a $\sigma$-algebra $\mathbb{A}$ and $m$ be a  measure on $(I,\mathbb{A})$, $m(I)=1$.  We will consider two measure spaces $(I,m)$ and $(J,\mu)$ along this paper, but often we will use $|A|$ for either  $m(A)$ or $\mu(A)$, since the measure under consideration will be clear from the context. 
 
 A {\bf grid}  is a sequence of finite  families of measurable sets with positive measure  $\mathcal{P}= (\mathcal{P}^k)_{k\in \mathbb{N}}$, so that at least one of these families is non empty and
\begin{itemize}
\item[G1.] Given $Q\in \mathcal{P}^k$, let
$$\Omega_Q^k=\{ P \in \mathcal{P}^k\colon \ P\cap Q\neq \emptyset \}.$$
Then
$$\Cll{mult1}=\sup_k \sup_{Q \in \mathcal{P}^k} \# \Omega_Q^k< \infty.$$
\end{itemize}

  Define  $||\mathcal{P}^k||=\sup \{|Q|\colon Q \in \mathcal{P}^k\}$.

A $(\Cll[namec]{G1},\Cll[namec]{G2})$-{\bf good  grid} , with $0< \Crr{G1}<  \Crr{G2} <  1$, is a grid  $\mathcal{P}= (\mathcal{P}^k)_{k\in \mathbb{N}}$  with the following properties:

\begin{itemize}
\item[G2.] We have $\mathcal{P}^0=\{I\}$.
\item[G3.] We have $I=\cup_{Q \in \mathcal{P}^k} Q$  (up to a set of zero $m$-measure).
\item[G4.] The elements of the  family $\{  Q \}_{Q\in \mathcal{P}^k} $ are pairwise disjoint. 
\item[G5.] For every $Q\in \mathcal{P}^k$ and $k > 0$ there exists $P \in \mathcal{P}^{k-1}$ such that $Q\subset P$.  
\item[G6.] We have
$$\Crr{G1} \leq \frac{|Q|}{|P|} \leq \Crr{G2}$$
for every $Q\subset P$ satisfying  $Q\in \mathcal{P}^{k+1}$ and $P\in \mathcal{P}^{k}$ for some $k\geq 0$. 
\item[G7.]  The family $\cup_k \mathcal{P}^k$ generate the  $\sigma$-algebra $\mathbb{A}$. 
\end{itemize}

We will  often abuse notation replacing $Q\in \cup_k\mathcal{P}^k$ by $Q\in \mathcal{P}$. For every set  $\Omega$, let 
$$k_0(\Omega,\mathcal{P})=\min \{k\geq 0 \colon \ \exists P \in \mathcal{P}^k \ s.t. \ P \subset \Omega     \} $$
whenever the set in the r.h.s. is a nonempty set. We will use the simpler notation $k_0(\Omega)$ if the grid under consideration is obvious.  Note that $k_0(W)=i$ for every $W \in \mathcal{P}^i$.

\section{Spaces defined by Souza's atoms} \label{space} Let $p \in [1,\infty]$ , $q \in [1,\infty)$ and $s \in (0,1)$ satisfying
$$0< s\leq \frac{1}{p}.$$
Let $\mathcal{P}=(\mathcal{P}^k)_{k\geq 0}$ be a good nested family of partitions. For every $Q\in \mathcal{P}$ let  $a_Q$ be   the function defined by $a_Q(x)=0$ for every $x \not\in Q$ and 
$$a_Q(x)=|Q|^{s-1/p}$$
for every $x \in Q$. The function $a_Q$ will be called a Souza's canonical atom on $Q$.  Let $\mathcal{B}^s_{p,q}$  be the space of all complex valued functions $f \in L^p$ that  can  be represented by an absolutely convergent series on $L^p$
\begin{equation} \label{rep55} f = \sum_{k=0}^{\infty} \sum_{Q \in \mathcal{P}^k}    s_Q a_Q\end{equation}
where $s_Q \in \mathbb{C}$ and additionally 
\begin{equation} \label{rep2} \big( \sum_{k=0}^{\infty} (\sum_{Q \in \mathcal{P}^k}    |s_Q|^p)^{q/p} \big)^{1/q} < \infty.\end{equation}
By absolutely convergence in $L^p$ we mean that 
\begin{equation} \sum_{k=0}^{\infty} \big| \sum_{Q \in \mathcal{P}^k}    s_Q a_Q  \big|_p < \infty. \end{equation} 
The r.h.s. of (\ref{rep55}) is called a $\mathcal{B}^{s}_{p,q}$-representation of $f$. Define
\begin{equation} \label{norm} |f|_{\mathcal{B}^s_{p,q}} = \inf  \big( \sum_{k=0}^{\infty} (\sum_{Q \in \mathcal{P}_k}     |s_Q|^p )^{q/p} \big)^{1/q} ,\end{equation} 
where the infimum runs over all possible representations of $f$ as in (\ref{rep55}).  We say that $f\in \mathcal{B}^s_{p,q}$ is  $\mathcal{B}^s_{p,q}$-positive if there is a $\mathcal{B}^{s}_{p,q}$-representation of $f$ such that $s_Q\geq 0$ for every $Q\in \mathcal{P}$.  The following results were proven in S. \cite{smania-besov}. We collect them here for the convenience of the reader. 

\begin{proposition}  The normed space $(\mathcal{B}^s_{p,q},|\cdot|_{\mathcal{B}^s_{p,q}})$ is  a complex Banach space  and its unit ball is compact in $L^p.$ 
\end{proposition} 

\begin{proposition}\label{inc}  We have that 
$$\mathcal{B}^s_{p,q}\subset L^{t}$$ 
for every $t$ satisfying 
$$p\leq t < \frac{p}{1-sp}.$$
Moreover this inclusion is continuous, that is, there is $K_t> 0$ such that 
$$|f|_{t}\leq K_t  |f|_{\mathcal{B}^s_{p,q}}.$$
for every $f\in \mathcal{B}^s_{p,q}$.
\end{proposition} 

There are many alternative definitions for $\mathcal{B}^s_{p,q}$. For instance, we can consider far more general atoms. Let
$$0< s< \beta < 1/p$$
Given $P\in \mathcal{P}$, denote by $\mathcal{B}^\beta_{p,q}(P)$ the set of all function $f \in \mathcal{B}^\beta_{p,q}$ that has a representation as in (\ref{rep55}) and (\ref{rep2}) and additionally $d_Q=0$ for every $Q\in \mathcal{P}$ that is not contained in $Q$.  The norm $|\cdot|_{\mathcal{B}^\beta_{p,q}(Q)}$ in $\mathcal{B}^\beta_{p,q}(Q)$ is defined as in (\ref{norm}), but the infimum is taken over all possible representations satisfying this additional condition. A $(s,\beta,p,q)$-Besov atom supported on $Q$ is a function $b_Q\in \mathcal{B}^\beta_{p,q}(Q)$ satisfying 
$$|b_Q|_{\mathcal{B}^\beta_{p,q}(Q)}\leq \frac{1}{\Cll[name]{GBVA}}|Q|^{s-\beta}.$$
We denote $\mathcal{A}^{bv}_{s,\beta,p,q}(Q)$ the set of $(s,\beta,p,q)$-Besov atoms supported on $Q$. The constant $\Crr{GBVA}$ is chosen such that  $a_Q\in \mathcal{A}^{bv}_{s,\beta,p,q}(Q).$

A  $(s,\beta,p,q)$-Besov {\it positive } atom supported on $Q$ is a function $b_Q\in \mathcal{A}^{bv}_{s,\beta,p,q}(Q)$ that has a  $\mathcal{B}^\beta_{p,q}(Q)$-representation 
$$b_Q= \sum_{k=0}^{\infty} \sum_{P \in \mathcal{P}^k, P\subset Q}    s_P a_P  $$
with $s_P\geq 0$ and satisfying 
$$\big( \sum_{k=0}^{\infty} (\sum_{Q \in \mathcal{P}_k}     |s_Q|^\beta )^{q/p} \big)^{1/q}\leq \frac{1}{\Cll[name]{GBVA}}|Q|^{s-\beta}.$$
The space $\mathcal{B}^{s}_{p,q}(\mathcal{A}^{bv}_{s,\beta,p,q})$ is the space of all functions that can be written as 
$$ f = \sum_{k=0}^{\infty} \sum_{Q \in \mathcal{P}^k}    s_Q b_Q$$
where $s_Q \in \mathbb{C}$, $b_Q\in  \mathcal{A}^{bv}_{s,\beta,p,q}$  and additionally  (\ref{rep2}) holds.  This is called a $\mathcal{B}^{s}_{p,q}(\mathcal{A}^{bv}_{s,\beta,p,q})$-representation of $f$. The norm $|\cdot|_{\mathcal{B}^{s}_{p,q}(\mathcal{A}^{bv}_{s,\beta,p,q})}$ is defined as in (\ref{norm}), where the infimum is taken instead over all possible $\mathcal{B}^{s}_{p,q}(\mathcal{A}^{bv}_{s,\beta,p,q})$-representations of $f$. Quite surprisingly we have

\begin{proposition}[From Besov to Souza representation] \label{besova} The Banach spaces  $\mathcal{B}^{s}_{p,q}(\mathcal{A}^{bv}_{s,\beta,p,q})$ and $\mathcal{B}^{s}_{p,q}$ coincides and its norms are equivalent. Indeed for every $\mathcal{B}^s_{p,q}(\mathcal{A}^{bv}_{s,\beta,p,q})$-representation  
$$g= \sum_{Q\in \mathcal{P}} d_Q  b_Q$$
there is a $\mathcal{B}^s_{p,q}(\mathcal{A}^{sz}_{s,p})$-representation 
$$g= \sum_{Q\in \mathcal{P}} z_Q  a_Q$$
such that 
\begin{equation*}
\Big(  \sum_i \Big(  \sum_{Q\in \mathcal{P}^i} |z_Q|^p    \Big)^{q/p}  \Big)^{1/q} \leq  \Cll[name]{GBS} \Big( \sum_i  \big( \sum_{W\in \mathcal{P}^i}  |s_W|^p  \big)^{q/p}\Big)^{1/q}.
\end{equation*}
Moreover if $d_Q\geq 0$ for every $Q\in \mathcal{P}$  and every  $b_Q$ is a  $(s,\beta,p,q)$-Besov positive   atom  supported on $Q$ then we can choose $z_Q\geq 0$ for every $Q\in \mathcal{P}$. 
\end{proposition}

\begin{proposition}\label{compactness}  The following sets are compact in $L^p$ (and in particular in $L^1$). 
\begin{itemize}
\item[A.] The set $S_C^{sz}$ of all  $f\in \mathcal{B}^s_{p,q}$ which have a $\mathcal{B}^s_{p,q}$-representation
\begin{equation} \label{eqww} f=\sum_k\sum_{Q\in \mathcal{P}^k} d_Q b_Q\end{equation} 
satisfying 
\begin{equation}\label{eqww2}  \sum_{k} (  \sum_{Q \in \mathcal{P}_k}     |d_Q|^p )^{q/p} \leq C.\end{equation} 
\item[B.] The  set $(S_C^{sz})^+$ of all $f\in \mathcal{B}^s_{p,q}$ that have a $\mathcal{B}^s_{p,q}$-representation satisfying (\ref{eqww}) and (\ref{eqww2}) and additionally $d_Q\geq 0$ for every $Q\in \mathcal{P}$.
\item[C.]  The set $S_C^{bv}$ of all  $f\in \mathcal{B}^s_{p,q}(\mathcal{A}^{bv}_{\beta,p,q})$ which has a $\mathcal{B}^s_{p,q}(\mathcal{A}^{bv}_{\beta,p,q})$-representation
\begin{equation} \label{eqw} f=\sum_k\sum_{Q\in \mathcal{P}^k} d_Q b_Q\end{equation} 
satisfying 
\begin{equation}\label{eqw2}  \sum_{k} (  \sum_{Q \in \mathcal{P}_k}     |d_Q|^p )^{q/p} \leq C.\end{equation}
\item[D.] The  set $(S_C^{bv})^+$ of all $f\in \mathcal{B}^s_{p,q}(\mathcal{A}^{bv}_{\beta,p,q})$ that has a $\mathcal{B}^s_{p,q}(\mathcal{A}^{bv}_{\beta,p,q})$-representation satisfying (\ref{eqw}) and (\ref{eqw2}) and additionally $d_Q\geq 0$ for every $Q\in \mathcal{P}$ and $b_Q$ is a   $(s,\beta,p,q)$-Besov positive   atom.
\end{itemize}
\end{proposition} 



\begin{proposition}[Canonical representation] \label{canrep2} There is $\Cll[name]{GC}\geq 1$ with the following property. For every $P\in \mathcal{P}$ there is a   linear functional  in $(L^1)^\star$
$$f\in L^1 \mapsto k_P^f$$
such that if $f\in \mathcal{B}^s_{p,q}$  then 
$$\sum_k \sum_{P\in \mathcal{P}^k} k_P^f a_P$$
is a $\mathcal{B}^s_{p,q}$-representation of $f$ satisfying
$$\Big( \sum_k \big( \sum_{P\in \mathcal{P}^k} |k_P^f|^p \big)^{q/p} \Big)^{1/q}\leq \Crr{GC} |f|_{\mathcal{B}^s_{p,q}}.$$
\end{proposition}

\begin{proposition}[From Besov to Souza representation] \label{besova} For every $\mathcal{B}^s_{p,q}(\mathcal{A}^{bv}_{\beta,p,q})$-representation  
$$g= \sum_{Q\in \mathcal{P}} d_Q  b_Q$$
there is a $\mathcal{B}^s_{p,q}(\mathcal{A}^{sz}_{s,p})$-representation 
$$g= \sum_{Q\in \mathcal{P}} z_Q  a_Q$$
such that 
\begin{equation*}
\Big(  \sum_i \Big(  \sum_{Q\in \mathcal{P}^i} |z_Q|^p    \Big)^{q/p}  \Big)^{1/q} \leq  \Cll[name]{GBS} \Big( \sum_i  \big( \sum_{W\in \mathcal{P}^i}  |s_W|^p  \big)^{q/p}\Big)^{1/q}.
\end{equation*}
Moreover if $d_Q\geq 0$ for every $Q\in \mathcal{P}$  and every atom $b_Q$ is $\mathcal{B}^s_{p,q}$-positive  then we can choose $z_Q\geq 0$ for every $Q\in \mathcal{P}$. 
\end{proposition} 

\fcolorbox{black}{white}{
\begin{minipage}{\textwidth}
\noindent {\bf Assumption $\Cll[A]{A000}$.} From  now on we fix measure spaces with  good grids $(I,\mathcal{P},m)$, $(J,\mathcal{H},\mu)$, $p\in [1,\infty)$, $q\in [1,\infty]$,  $\epsilon > 0$, $\gamma\in [0,1]$  and $s, \beta,\delta \in (0,\infty)$ be such that
$$0< s+\epsilon  \leq  \frac{1}{p}, \ s< \beta < \frac{1}{p}$$
and
$0< \delta < \max \{ s,\epsilon\}.$ 
\end{minipage} 
}
\ \\ \\
Let
\begin{equation} \label{t0} t_0=\frac{p}{1-sp+\delta p}.\end{equation} 
Note that 
$$p<  t_0 < \frac{p}{1-sp},$$so in n particular $\mathcal{B}^s_{p,q}\subset L^{t_0}$.

\begin{table}[h]
  \centering
  \caption{Constants associated with the {\bf G}eometry of the grid}
  \label{tab:table2}
  \begin{tabular}{cc}
    \toprule
    Symbol & Description\\
    \midrule
    $\Crr{G1}\leq \Crr{G2}$ &{\bf G}eometry of the grid\\
    $\Crr{GC}$  & Describes how optimal is the {\bf C}anonical Souza's representation \\
    $\Crr{GBS}$& Control the conversion of a representation using {\bf B}esov's atoms to \\
    & a representation using {\bf S}ouza's atoms. \\
    \bottomrule
  \end{tabular}
\end{table}

\section{Regular domains}\label{rdf}

   We say that $\Omega\subset J$ is a $(\alpha, \Cll{domain},\Cll[c]{DRegDom})$-regular domain if  it is possible to find  families $\mathcal{F}^k(\Omega)\subset \mathcal{H}^k$, $k\geq  k_0(\Omega,\mathcal{H})$,  such that

\begin{itemize}
\item[A.] We have $\Omega = \cup_{k\geq k_0(\Omega)} \cup_{Q\in \mathcal{F}^k(\Omega)} Q$.
\item[B.] If $P,Q \in \cup_{k\geq k_0(\Omega)} \mathcal{F}^k(\Omega)$ and $P\neq Q$ then $P\cap Q=\emptyset$. 
\item[C.] We have
\begin{equation}\label{dom}  \sum_{Q\in \mathcal{F}^k(\Omega)} |Q|^{\alpha}\leq \Crr{domain} \Crr{DRegDom}^{k-k_0(\Omega)} |\Omega|^{\alpha}.\end{equation} 
\end{itemize}

\section{Branches}\label{contra}  Let $\hat{I}\subset I$ and $\hat{J}\subset J$ be measurable sets. A 
$$(s,p,a, \tilde{\epsilon}, \Cll[name]{DGD1},\Cll[namec]{DGD2}, \Cll[name]{DC1},\Cll[namec]{DC2},\mathcal{G})-branch$$ is a measurable  function
$$h\colon \hat{J} \rightarrow \hat{I}$$
such that 
$$h^{-1}\colon \hat{I} \rightarrow \hat{J}$$
is also measurable, $\mathcal{G}$ is a grid on $(\hat{I},m)$  and \\
\begin{itemize}
\item[I.] We have that  $m(Q)=0$ if and only if $\mu(h^{-1}(Q))=0$ for every measurable  set $Q\subset J$.\\
\item[II.] {\bf (Geometric Distortion Control).} For every $Q\subset I$ such that $Q\in \mathcal{G}$  we have  that $h^{-1}(Q)$ is a $(1-sp,\Crr{DGD1},\Crr{DGD2})$-regular domain in $(J,\mathcal{H})$. This property controls how the action of $h^{-1}$ {\it deforms} the ''shape"  of $Q$. \\
\item[III.]{\bf (Scaling Control).} We have 
$$|k_0(Q,\mathcal{G})-k_0(h^{-1}(Q),\mathcal{H})|  \geq   a.$$
and 
\begin{equation} \label{tam} \Big(  \frac{|Q|}{|h^{-1}(Q)|}\Big)^{\tilde{\epsilon}/|\tilde{\epsilon}|} \leq \Crr{DC1} \Crr{DC2}^{ |k_0(Q,\mathcal{G})-k_0(h^{-1}(Q),\mathcal{H})|}.\end{equation}  
\item[IV.] The set $\hat{I}$   is a countable union (up to a set of zero measure) of elements of $\mathcal{P}$.\\
\item[V.]  For every $W\in \mathcal{G}$ such that $W \subset \hat{I}$ we have that $h^{-1}(W)$ is a countable union (up to a set of zero measure) of elements of $\mathcal{H}$.
\end{itemize}


\section{Potentials}\label{pot} Let 
$$h\colon \hat{J} \rightarrow \hat{I}$$
be a branch  as in Section \ref{contra}. A $(\Cll[name]{DRP}, \beta,\tilde{\epsilon})$-{\bf  regular potential}, with $\beta$ such that $s < \beta\leq 1/p$, associate to $h$ is a function $g\colon \hat{J} \rightarrow \mathbb{C}$  satisfying 
\begin{equation}\label{supe33}  |g\cdot 1_{W}|_{\mathcal{B}^\beta_{p,q}(W,\mathcal{H}_W,\mathcal{A}^{sz}_{p,q})}    \leq \Crr{DRP}  \Big( \frac{|Q|}{|h^{-1}Q|}\Big)^{\frac{1}{p} -s+\tilde{\epsilon}} |W|^{1/p-\beta}.\end{equation}
for every  $W \in \mathcal{H}$ and $Q \in \mathcal{G}$ such that $W \subset \hat{J}$, $Q \subset \hat{I}$ and $h(W)\subset Q$.

We say the potential $g$ is  $\mathcal{B}^s_{p,q}$-{\bf positive regular potential}  if for every $W\in \mathcal{H}$ such that $W \subset J$, there is a $\mathcal{B}^\beta_{p,q}$-representation of $g\cdot 1_{W}$
$$g \cdot 1_{W}=\sum_k  \sum_{\substack{P\in \mathcal{H}^k\\Q\subset W}} c_P a_P$$
such that $c_P\geq 0$ for every $P\in \mathcal{H}$ and moreover

\begin{equation}  \Big( \sum_k \big( \sum_{\substack{P\in \mathcal{H}^k\\P\subset W}}   |c_P|^p \big)^{q/p} \Big)^{1/q}  \leq \Crr{DRP}\Big( \frac{|Q|}{|h^{-1}Q|}\Big)^{\frac{1}{p} -s+\tilde{\epsilon}} |W|^{1/p-\beta}.\end{equation}
for every  $W\in \mathcal{H}$ and $Q \in \mathcal{G}$ such that $W \subset J$, $Q \subset I$ and $h(W)\subset Q$.

\section{Transfer transformations}  \label{transfer}
Assume
\ \\  \\
\fcolorbox{black}{white}{
\begin{minipage}{\textwidth}
\noindent {\bf Assumption $\Cll[A]{A00}$.} Along this paper we will always assume that 
\begin{itemize}
\item $\{I_r\}_{r\in \Lambda}$ and $\{J_r\}_{r\in \Lambda}$ are   families of measurable subsets of  $I$ and $J$, respectively,  with $\Lambda\subset \mathbb{N}$. 
\item The maps 
$$h_r\colon J_r \rightarrow I_r$$
are  $(s,p,a_r, \epsilon_r,\Crr{DGD1}^r, \Crr{DC1}^r,\Crr{DGD2}^r,\Crr{DC2}^r,\mathcal{G}_r)$-branches, with $|\epsilon_r|=\epsilon$ for every $r$.
\item  We have that 
$$\mathcal{A}= \{Q \in \mathcal{G}_r, \  Q\subset I_r, \text{ for some $r\in \Lambda$}\}\cup \{Q \in \mathcal{P}, \  Q\cap I_r=\emptyset, \text{ for every $r\in \Lambda$}\}.$$
generates the $\sigma$-algebra $\mathbb{A}$. 
\item We have that 
$$g_r\colon J_r \rightarrow I_r$$
are  $(\Crr{DRP}^r, \beta,\epsilon_r)$-potentials.\\
\item Let 
$$\Cll[namec]{DRS2}^r=\max \{ (\Crr{DC2}^r)^{\epsilon},(\Crr{DGD2}^r)^{1/p}\} < 1.$$
Then  $\Crr{DRS2}= \sup_r \Crr{DRS2}^r < 1.$
\end{itemize} 
 The ``$r$" in the notation $\Crr{DGD1}^r, \Crr{DC1}^r$, $\Crr{DGD2}^r$,$\Crr{DC2}^r$ and $\Crr{DRP}^r$ {\it is just an index}, indicating that those constants  may depend on $r$.
\end{minipage} 
}
\ \\ \\ 
The value
$$\Theta_r=(\Crr{DC1}^r)^{\epsilon} \Crr{DRP}^r  ( \Crr{DGD1}^r)^{1/p}  ( \Crr{DRS2}^r)^{a_r(1-\gamma)}$$
 measures how regular is the $r$-th pair $(h_r,g_r)$.\\

We want to  consider the  {\bf transfer transformation} 
$$\Phi(f)(x)= \sum_{r\in \Lambda} g_r(x)f(h_r(x)).$$
Notice that  when $\Lambda$ is an infinity set it is not even clear for which measurable functions $f$ the operator is well defined.  Let 
$$\Phi_r(f)(x)= g_r(x)f(h_r(x)).$$
We also assume \\ \\
\fcolorbox{black}{white}{
\begin{minipage}{\textwidth}
\noindent {\bf Assumption $\Cll[A]{A0}$.} There is  $\Cll{232}\geq0$ such that for every $f\in L^{t_0}(m)$
$$|\Phi f|_{L^{1}(\mu)}\leq \sum_r |\Phi_r(f)|_{L^{1}(\mu)}\leq \Crr{232}|f|_{L^{t_0}(m)}.$$
\end{minipage} 
}
\ \\ 
Section \ref{boundlp} provides  some methods  to obtain Assumption $\Crr{A0}$. Note that  Assumption $\Crr{A0}$  implies  that  $\Phi\colon L^{t_0}(m)\mapsto L^1(\mu)$ is a well defined and  bounded linear transformation, where $t_0$ is defined in (\ref{t0}).  Then  
$$p\leq t_0 < \frac{p}{1-sp}$$
and  Proposition \ref{inc} implies that  
$$\Phi\colon \mathcal{B}^s_{p,q}(I,\mathcal{P},m) \rightarrow L^{1}(\mu)$$ is a bounded linear transformation.

\section{Regular dynamical slicing}    We want give conditions for  $\Phi$ to be a well-defined linear transformation from $\mathcal{B}^s_{p,q}(I,m,\mathcal{P})$ to  $\mathcal{B}^s_{p,q}(J,\mu,\mathcal{H})$ and study its regularity. To this end  we need to connect the ``local" grid $\mathcal{G}_r$ on each  $(I_r,m)$ with the "global"  good grid $\mathcal{P}$ in $(I,m)$. This  will depend  on the data
$$\mathcal{I}=(s,p,q, \epsilon, \gamma,  \{(I_r,J_r, a_r,\Crr{DGD2}^r, \Crr{DC2}^r,\Crr{DC1}^r, \Crr{DRP}^r,\Crr{DGD1}^r,\mathcal{G}_r)\}_{r\in \Lambda}).$$
We call $\mathcal{I}$ a {\bf weighed family of sets}. Let 
\begin{equation}\label{ND} N=\sup_{P\in \mathcal{H}} \#\{ r\in \Lambda \ s.t. \  P\subset J_r \}.\end{equation}
We say that the pair
$$(\mathcal{I}, \sum_k \sum_{Q\in \mathcal{P}^k}  d_Q   a_Q)$$
has a  {\bf $\Cll[name]{DRS1}$-regular slicing}  if $\Crr{DRS1}\geq 0$ and  
\begin{itemize}
\item[i.] We have that 
$$f=\sum_k \sum_{Q\in \mathcal{P}^k}  d_Q   a_Q$$
is a $\mathcal{B}^s_{p,q}(I,m,\mathcal{P})$-representation of a function $f \in \mathcal{B}^s_{p,q}(I,m,\mathcal{P})$.
\item[ii.]   For every $r \in \Lambda$  there is a $\mathcal{B}^s_{p,q}(I_r,\mu,\mathcal{G}_r)$-representation of  $f\cdot 1_{I_r}$
\begin{equation}\label{posi} f\cdot 1_{I_r}= \sum_{Q\in \mathcal{G}_r, Q \subset I_r }  c_Q^r   a_Q,\end{equation}
 satisfying either
\begin{align}\label{hiip1} &\Big( \sum_j\big( \sum_r  \Theta_r   (\sum_{\substack{ Q\in \mathcal{G}^j_r \\  Q \subset I_r }}  |c_Q^r|^p )^{1/p}  \big)^{q}\Big)^{1/q} \\
& \leq \Crr{DRS1}  \big( \sum_k (\sum_{Q \in \mathcal{P}^k}    |d_Q|^p)^{q/p} \big)^{1/q}, \nonumber  \end{align}
or
\begin{align}\label{hiip2}&N^{1/p'} \Big( \sum_j\big( \sum_r \Theta_r^p \sum_{\substack{ Q\in \mathcal{G}^j_r \\  Q \subset I_r }}  |c_Q^r|^p   \big)^{q/p}\Big)^{1/q} \\
& \leq \Crr{DRS1}  \big( \sum_k (\sum_{Q \in \mathcal{P}^k}    |d_Q|^p)^{q/p} \big)^{1/q}. \nonumber \end{align}
Here $N^{1'}=1$.
\item[iii.] If $d_Q\geq 0$ for every $Q\in \mathcal{P}$ then we can choose $c_Q^r\geq 0$ for every $Q\in \mathcal{G}_r$. 
\end{itemize} 

\begin{table}[h]
  \centering
  \caption{Constants associated with the {\bf D}ynamics of the transfer operator}
  \label{tab:table2}
  \begin{tabular}{cc}
    \toprule
    Symbol & Description\\
    \midrule
    $\Crr{DGD1}^r,\Crr{DGD2}$& Describes the {\bf G}eometric {\bf D}eformation  \\
    & of the domains in the grid by the action of branches $h_r$\\
     $a_r$, $\Crr{DC1}^r$, $\Crr{DC2}$ & Describe the {\bf C}ontracting  properties of $h_r$ \\
    $\Crr{DRP}^r$& Describes the {\bf R}egularity of the {\bf P}otentials $g_r$ \\
    $\Crr{DRS1}$ & Controls the {\bf R}egularity of the dynamical  {\bf S}licing.\\
    \bottomrule
  \end{tabular}
\end{table}
\newpage
\vspace{1cm}
\centerline{ \bf III. STATEMENT OF RESULTS.}
\addcontentsline{toc}{chapter}{\bf III. STATEMENT OF RESULTS.}
\vspace{1cm}

\section{Boundeness on $\mathcal{B}^s_{p,q}$} \label{bsspq}

Our main technical result is 
\begin{theorem}[Key Technical Result]\label{key} Let
$$ \sum_k \sum_{Q\in \mathcal{P}^k} d_Q a_Q$$
be a $\mathcal{B}^s_{p,q}$-representation of a function $ f\in \mathcal{B}^s_{p,q}(I,m,\mathcal{P})$ such that $(\mathcal{I}, \sum_{Q\in \mathcal{P}} d_Q a_Q)$ has a  $\Crr{DRS1}$-regular slicing.  Define
$$\Cll[name]{D}=  \frac{2}{1-\Crr{DRS2}^{\gamma}}.$$
Then $\Phi(f)\in \mathcal{B}^s_{p,q}(J,\mu,\mathcal{H})$ and there  is a $\mathcal{B}^s_{p,q}$-representation 
$$\sum_k \sum_{Q\in \mathcal{H}^k} z_Q a_Q$$
of $\Phi(f)$ such that 
\begin{align*}
&\Big(  \sum_i \Big(  \sum_{Q\in \mathcal{H}^i} |z_Q|^p    \Big)^{q/p}  \Big)^{1/q} 
\leq \Crr{GBS}\Crr{D} \Crr{DRS1}   \big( \sum_k (\sum_{Q \in \mathcal{P}^k}    |d_Q|^p)^{q/p} \big)^{1/q}. \end{align*}
Moreover if the potentials $g_r$ are $\mathcal{B}^s_{p,q}(J,\mu,\mathcal{H}$-positive then whenever $d_Q\geq0 $ for every $Q$ we can choose $z_Q \geq 0$ for every $Q$.
\end{theorem}

The proof of the following result is obvious.
\begin{corollary} Let $S$ be a linear subspace of  $\mathcal{B}^s_{p,q}(I,m,\mathcal{P})$. Suppose that for every $ f\in S$   there is a $\mathcal{B}^s_{p,q}$-representation
$$f= \sum_k \sum_{Q\in \mathcal{P}^k} d_Q a_Q$$
such that 
$$ \big( \sum_k (\sum_{Q \in \mathcal{P}^k}    |d_Q|^p)^{q/p} \big)^{1/q}\leq C |f|_{\mathcal{B}^s_{p,q}}$$
and $(\mathcal{I}, \sum_{Q\in \mathcal{P}} d_Q a_Q)$ has a  $\Crr{DRS1}$-regular slicing.  Then $$\Phi\colon S\rightarrow \mathcal{B}^s_{p,q}(J,\mu,\mathcal{H})$$ is a linear transformation satisfying 
$$|\Phi(f)|_{\mathcal{B}^s_{p,q}(J,\mu,\mathcal{H})}\leq  \Crr{GBS} \Crr{D}\Crr{DRS1} C   |f|_{\mathcal{B}^s_{p,q}(I,m,\mathcal{P})}.$$
for every $f\in S$.
\end{corollary} 

 \section{Dynamical Slicing: How to do it}
 
 \begin{definition}\label{c2}  Let $S$ be a subspace of $ \mathcal{B}^s_{p,q}(I,m,\mathcal{P})$ and let $\mathcal{I}$ be a weighed family of sets, We say that $(S,\mathcal{I})$ has a $(\Cll[name]{DRSFR},\Cll[name]{DRSES})$-essential slicing, where $\Crr{DRSFR}, \Crr{DRSES}\geq 0$ if there is a {\it finite} subset $\mathcal{P}'\subset \mathcal{P}$  such that for every 
$\mathcal{B}^s_{p,q}$-representation $$\sum_{Q\in \mathcal{P}} d_Q  a_Q$$  of a function $f \in S$
the pair 
$$(\mathcal{I},\sum_{Q\in \mathcal{P} \setminus \mathcal{P}'} d_Q  a_Q ),$$
has a  $\Crr{DRSES}$-regular slicing
and the pair 
$$(\mathcal{I},\sum_{Q\in \mathcal{P}'} d_Q  a_Q ),$$
has a $\Crr{DRSFR}$-regular slicing.  Here  FR stands for ``{\bf F}inite-{\bf R}ank"  and ES for ``{\bf ES}sential spectral radius".
We say that $(S,\mathcal{I})$ has a $(\Crr{DRSFR},\Crr{DRSES},t)$-essential slicing if $\mathcal{P}'=\cup_{k< t} \mathcal{P}^k.$
\end{definition} 

Given $\mathcal{R}\subset \mathcal{P}$, define the closed subspace $\mathcal{B}^s_{p,q,\mathcal{R}} \subset \mathcal{B}^s_{p,q}$ as
$$\mathcal{B}^s_{p,q,\mathcal{R}}=\{ f\in  \mathcal{B}^s_{p,q}\colon \   k^f_P=0 \text{ for every } P\not\in   \mathcal{R}  \}.$$
Here $k^f_P$ is as in Proposition \ref{canrep2}.   Note that there is a linear projection
$$\pi_{\mathcal{G}}\colon \mathcal{B}^s_{p,q} \rightarrow \mathcal{B}^s_{p,q,\mathcal{R}}$$
satisfying $|\pi_{\mathcal{R}}|\leq \Crr{GC}$ and moreover $f=\pi_{\mathcal{R}}(f)+\pi_{\mathcal{P}\setminus \mathcal{R}}(f)$ for every $f \in  \mathcal{B}^s_{p,q}(I,m,\mathcal{P})$. Of course $\mathcal{B}^s_{p,q,\mathcal{R}}$ has finite dimension when $\mathcal{R}$ is finite.



 We left unanswered how to obtain a regular dynamical slicing  as assumed in  our main results on transfer operators,  as Theorem \ref{key} and Corollary \ref{ww}.  This section deals with this question.

\begin{definition} As defined in S. \cite{smania-besov}, a  set $\Omega\subset I$ is {\bf $(\alpha,\Cll{rp},t )$-strongly  regular domain }  if  for each  $Q\in \mathcal{P}^i$, with $i\geq t$  and $k\geq k_0(Q\cap \Omega)$ there is  a family $\mathcal{F}^k(Q\cap \Omega) \subset \mathcal{P}^k$  such that 
\begin{itemize}
\item[i.] We have $Q\cap \Omega = \cup_{k\geq k_0(Q\cap \Omega)} \cup_{P\in \mathcal{F}^k(Q\cap \Omega)} P$.
\item[ii.] If $P,W \in \cup_{k} \mathcal{F}^k(Q\cap \Omega)$ and $P\neq W$ then $P\cap W=\emptyset$. 
\item[iii.] We have
\begin{equation} \sum_{P\in \mathcal{F}^k(Q\cap \Omega)} |P|^{\alpha}\leq \Crr{rp}  |Q|^{\alpha}.\end{equation} 
\end{itemize}
\end{definition} 

in the next four results we will assume \\

\fcolorbox{black}{white}{
\begin{minipage}{\textwidth}
\noindent {\bf Assumption PLAIN.} For every $r\in \Lambda$ 
$$\mathcal{G}_r^k=\{ P\in \mathcal{P}^k\colon \ P\subset I_r  \}.$$
\end{minipage} 
}

\begin{theorem}[The Core I] \label{pa} Assume $\Crr{A000}-\Crr{A0}$ and PLAIN. There is $\Cll[name]{GSR}$, that depends only the good grid $\mathcal{P}$, with the following property.   Suppose that $\Lambda$ is finite and  there  is $t$ such that for  every  with $r\in \Lambda$ the set $I_r$  is a $(1-\beta p,\Cll{rp1f},t)$-strongly regular domain. Suppose
\begin{equation}\label{m} M=\sup_{\substack{P\in \mathcal{P}^k \\ k\geq t}}\# \{r\in \Lambda \colon I_r\cap P\neq \emptyset    \} < \infty.\end{equation}
\noindent Then  $(\mathcal{B}^s_{p,q},\mathcal{I})$ has a $(\Crr{GSR}\Crr{f2f},\Crr{GSR}\Crr{f1f},t)$-essential slicing, with
\begin{align*}\Cll{f1f}&=  M \Crr{rp1f}^{1/p} \big( \sum_{r \in \Lambda}   \Theta_r^{p'}  \big)^{1/p'} .
\end{align*}
and
\begin{align*}\Cll{f2f}&= (\#\Lambda)   (\Crr{rp1f}  \Crr{G1}^{-t})^{1/p}  \big( \sum_{r \in \Lambda}   \Theta_r^{p'}  \big)^{1/p'}.
\end{align*}
\end{theorem}

\begin{theorem}[The Core II] \label{pa2} Assume $\Crr{A000}-\Crr{A0}$ and PLAIN. There is $\Crr{GSR}$, that depends only the good grid $\mathcal{P}$, with the following property.   Suppose that  $\Lambda$ is  finite and  there  is $t$ such that for  every  with $r\in \Lambda$ the set $I_r$  is a $(1-\beta p,\Crr{rp1f},t)$-strongly regular domain, 
\begin{equation}\label{t} T=\sup_{\substack{Q\in \mathcal{P}^k\\k\geq t}} \sum_{\substack{Q\cap I _r\neq \emptyset \\ r\in \Lambda}} \Theta_r < \infty,\end{equation}
and
\begin{equation}\label{n} N=\sup_{P\in \mathcal{H}} \#\{ r\in \Lambda \ s.t. \  P\subset J_r \}< \infty.\end{equation}

\noindent Then  $(\mathcal{B}^s_{p,q},\mathcal{I})$ has a $(\Crr{GSR}\Crr{f2a},\Crr{GSR}\Crr{f1a},t)$-essential slicing, with
\begin{align*}\Cll{f1a}&=  N^{1/p'} \Crr{rp1f}^{1/p}T,
\end{align*}
and
\begin{align*}\Cll{f2a}&= N^{1/p'}  (\#\Lambda)(\sup_{r\in \Lambda} \Theta_r)   (\Crr{rp1f}  \Crr{G1}^{-t})^{1/p},
\end{align*}
with the obvious adaptation when $p=1$ (in particular we set  $N^{1/\infty}=1$). 
\end{theorem}

\begin{theorem}[Tail I] \label{tt1} Assume $\Crr{A000}-\Crr{A0}$ and PLAIN. There is $\Cll[name]{GSR}$, that depends only the good grid $\mathcal{P}$, with the following property.  Suppose $\{ I_r \}$ is a family of pairwise disjoint  $(1-\beta p,\Cll{rp36},0)$-strongly regular domains. 
\noindent  If 
\begin{align*}\Cll{f2bb}&= \Crr{rp36}^{1/p}(\sum_{r\in \Lambda} \Theta_r),\end{align*}
 is finite then  $(\mathcal{B}^s_{p,q},\mathcal{I})$ has a $(\Crr{GSR}\Crr{f2bb})$-regular slicing.
\end{theorem}

\begin{theorem}[Tail II] \label{tt2} Assume $\Crr{A000}-\Crr{A0}$ and PLAIN. There is $\Cll[name]{GSR}$, that depends only the good grid $\mathcal{P}$, with the following property.  Suppose $\{ I_r \}_{r\in \Lambda}$ is a countable family of pairwise disjoint subsets  such that  the set
$$\Omega= \cup_{i\in \Lambda} I_r$$
is a $(1-\beta p,\Crr{rp36},0)$-strongly regular domain and additionally, if $Q\in \mathcal{P}$ and $Q\subset \Omega$ then $Q\subset I_r$, for some $r\in \Lambda$. 
Let  $$N=\sup_{P\in \mathcal{H}} \#\{ r\in \Lambda \ s.t. \  P\subset J_r \}.$$
\noindent If 
\begin{align*}\Cll{f2b}&= N^{1/p'} \Crr{rp36}^{1/p}(\sup_{r\in \Lambda} \Theta_r),\end{align*}
 is finite then  $(\mathcal{B}^s_{p,q},\mathcal{I})$ has a $(\Crr{GSR}\Crr{f2b})$-regular slicing. Here  we set  $N^{1/\infty}=1$ {\bf even} if $N=\infty$. 
\end{theorem}

\section{Essential Spectral Radius  of $\Phi$ acting on $\mathcal{B}^s_{p,q}$}


Consider the assumption
\vspace{5mm}

\fcolorbox{black}{white}{
\begin{minipage}{\textwidth}
\noindent {\bf Assumption $\Cll[A]{A1}.$}  We have $(J,\mu,\mathcal{H})=(I,m,\mathcal{P})$, and  $(\mathcal{B}^s_{p,q}(I,m,\mathcal{P}),\mathcal{I})$ has a $(\Crr{DRSFR},\Crr{DRSES})$-essential slicing.
\end{minipage} 
}
\vspace{5mm}

We stress we are not assuming  PLAIN anymore.

\begin{corollary}[Boundedness and Essential Spectrum Radius of $\Phi$] \label{ww} Suppose that $\Crr{A000}-\Crr{A1}$ hold. Then  the  transformation
 $\Phi$ is a bounded  operator acting on $\mathcal{B}^s_{p,q}$ satisfying 
$$|\Phi|_{\mathcal{B}^s_{p,q}}\leq \Crr{GBS} \Crr{D} (\Crr{DRSFR}+\Crr{DRSES})\Crr{GC} $$
and its essential espectral radius  is at most  $  \Crr{GBS} \Crr{D}    \Crr{DRSES}\Crr{GC} .$
Indeed
$$\Phi\circ \pi_{\mathcal{P}'}\colon \mathcal{B}^s_{p,q}\rightarrow \mathcal{B}^s_{p,q}$$
is a finite-rank linear  transformation with norm at most $  \Crr{GBS} \Crr{D} \Crr{DRSFR}\Crr{GC}$ and
$$\Phi\circ \pi_{\mathcal{P}\setminus \mathcal{P}'}\colon \mathcal{B}^s_{p,q}\rightarrow \mathcal{B}^s_{p,q}$$
is a  linear  transformation with norm at most $ \Crr{GBS} \Crr{D} \Crr{DRSES}\Crr{GC}$.
 \end{corollary}

\section{Lasota-Yorke Inequality and its consequences} 
\label{lyorke} 

We have 

\begin{theorem}[Lasota-Yorke Inequality] \label{t1}Suppose that $\Crr{A000}-\Crr{A1}$  holds, $$  \Crr{GBS} \Crr{D}  \Crr{DRSES}\Crr{GC}   < 1$$ and
$$|\Phi(f)|_1\leq |f|_1$$
for every $f\in \mathcal{B}^s_{p,q}$.
Then $\Phi$ satisfies the Lasota-Yorke inequality 
$$|\Phi^n(f)|_{\mathcal{B}^s_{p,q}}\leq C |f|_1+(  \Crr{GBS} \Crr{D}  \Crr{DRSES}\Crr{GC} )^n|f|_{\mathcal{B}^s_{p,q}}.$$
for some $C\geq 0$ and every $f\in \mathcal{B}^s_{p,q}$.
\end{theorem}

Consider

\vspace{5mm}
\fcolorbox{black}{white}{
\begin{minipage}{\textwidth}
\noindent {\bf Assumption $\Cll[A]{A5}.$}  We have
$$6  \Crr{GBS} \Crr{D}   \Crr{DRSES}\Crr{GC}   < 1.$$ 
Moreover the potentials satisfy    $g_r\geq 0$ $m$-almost everywhere,  and
$$\int \Phi(f) \ dm= \int f \ dm$$
for every $f\in \mathcal{B}^s_{p,q}(I,m,\mathcal{P})$. \end{minipage} 
}

\ \\ \\
Note that Assumption $\Crr{A5}$ implies that $|\Phi(f)|_1\leq |f|_1$ for every $f\in \mathcal{B}^s_{p,q}$. Since $\mathcal{B}^s_{p,q}$ is dense in $L^1$  we can extend $\Phi$ to a bounded linear operator $\Phi\colon L^1 \rightarrow L^1$. The multiplicative factor $6$ is only needed in the proof of Corollary \ref{asip1}.



\begin{corollary}\label{spec} Suppose that  $\Crr{A000}-\Crr{A5}$  hold. Then 
\begin{itemize} 
\item[A.] There exists $\delta \in (0,1)$ such that    $$E= \sigma_{\mathcal{B}^s_{p,q}}(\Phi)\cap \{z \in \mathbb{C}\colon\ |z| \geq \delta  \}$$ is finite and nonempty and it is contained  in $\mathbb{S}^1$.  Every element $\lambda \in E$ is an eigenvalue with finite-dimensional eigenspace  and 
$$(\Phi-\lambda I)^2f=0 \ implies \ (\Phi-\lambda I)f=0$$
for every $f\in \mathcal{B}^s_{p,q}$
\item[B.] For every $\lambda \in E$ there is a bounded projection $\pi_\lambda$, and there is  a linear contraction $\tilde{\Theta}$, both of them acting on $\mathcal{B}^s_{p,q}$, such that
$$\Phi^n = \sum_{\lambda \in E} \lambda^n \pi_\lambda + \tilde{\Theta}^n$$
and
$$\Phi \pi_\lambda = \lambda \pi_\lambda, \  \pi_{\lambda'}\pi_\lambda = 0,\   \tilde{\Theta}\pi_\lambda = \pi_\lambda \tilde{\Theta}=0  $$
for every $\lambda',\lambda \in E$, $\lambda \neq \lambda'$. 
In particular $\sup_n |\Phi^n|_{\mathcal{B}^s_{p,q}} < \infty$.
\item[C.] There is  $\rho \in \mathcal{B}^s_{p,q}$, with $\rho \geq 0$, such that  
\begin{equation}\label{co}  \int \rho \ dm=1\  and \ \Phi(\rho)=\rho.\end{equation} 
\item[D.] If $\lambda \in \mathbb{S}^1$ and $u \in L^1$ satisfies $\Phi u = \lambda u$ then $u\in \mathcal{B}^s_{p,q}$. Moreover
$\Phi(|u|)=|u|$ and  $u_k = (sgn \ u)^k |u|$ satisfies $\Phi u_k =\lambda^k u_k$ for every $k \in \mathbb{Z}$. 
\item[E.] Every element of the finite set $E$ is a $n$-th root of unit, for  some $n \in \mathbb{N}^\star$. 
\end{itemize}
\end{corollary}

 \begin{corollary} \label{phy} Suppose $\Crr{A000}-\Crr{A5}$. Then $T$ has  at most $N$  physical  measures, where $N$ is the dimension of the $1$-eigenspace of $\Phi$. Moreover all these measures are absolutely continuous with respect to $m$, and the basin of attraction of these measures covers $I$.
 \end{corollary}

Denote by $M(\mathcal{B}^s_{p,q}) \subset \mathcal{B}^s_{p,q}$ the set of the functions $g$ such that the pointwise multiplication   
$$f \mapsto fg $$
is a bounded operator in $\mathcal{B}^s_{p,q}$, that is, if $f\in \mathcal{B}^s_{p,q}$ then $fg \in \mathcal{B}^s_{p,q}$ and 
$$\sup \{ |fg|_{\mathcal{B}^s_{p,q}}\colon \ |f|_{\mathcal{B}^s_{p,q}}\leq 1 \} < \infty. $$
We know that $B^{1/p}_{p,\infty}\cap L^\infty\subset M(\mathcal{B}^s_{p,q})\subset L^\infty$ (see \cite{smania-besov}).
\begin{corollary}[Central Limit Theorem and Almost Sure Invariance Principle] \label{asip1} Suppose that  $\Crr{A000}-\Crr{A5}$  hold for some $p \in [1,\infty)$, $q\in [1,\infty)$ and $0< s< 1/p$ 
and that  $1$ is a simple eigenvalue of the transfer operator $\Phi\colon B^{s}_{p,q}\rightarrow B^{s}_{p,q}$. In particular there is a unique absolutely continuous invariant probability $\rho_0 m$, and $\rho_0 \in B^s_{p,q}$. Then for every  function $v=v_1+v_2$, where $v_1  \in \mathcal{B}^{1/p}_{p,\infty}$ and $v_2\in M(\mathcal{B}^s_{p,q})$ are real-valued functions, such that  
\begin{equation}\label{average}  \int v  \rho_0  dm=0\end{equation}
we have that 
\begin{itemize}
\item[A.] the following limit 
\begin{equation}\label{sigma} \sigma^2(v) = \lim_{n\rightarrow \infty} \int  \Big(\frac{ \sum_{i=0}^{n-1} v\circ T^i}{\sqrt{n}}\Big)^2 \ \rho_0 dm\end{equation}
converges.
\item[B.] if  $\sigma(v)> 0$ the sequence 
\begin{equation}\label{series} v , v \circ T, v \circ T^2...\end{equation}
satisfies the Central Limit Theorem with average $0$ and variance $\sigma$ and the Almost Sure Invariance Principle with every  error exponent satisfying
$$\delta >\frac{1}{4}.$$
\end{itemize}
\end{corollary}

\section{Uniqueness and structure of invariant measures} 
\label{lyorke2}

The property   that $1$  is a simple eigenvalue  is quite useful and sometimes the trickiest to achieve, particularly in dynamical systems with discontinuities. The following  assumption  allows us to better  understand the structure  of the invariant densities and it makes it easier to check $E=\{1\}$.  \\

\fcolorbox{black}{white}{
\begin{minipage}{\textwidth}
\noindent {\bf Assumption $\Cll[A]{A8}$.} The potentials $g_r$ are $\mathcal{B}^s_{p,q}$-positive.
\end{minipage} 
}
\ \\

\begin{theorem}[Structure of invariant measures] \label{acim} Suppose $\Crr{A000}-\Crr{A8}$. Then  every $\rho \in L^1$, $\rho\geq 0$ that satisfies (\ref{co}) is  $\mathcal{B}^s_{p,q}(I,m,\mathcal{P})$-positive. In particular  
\begin{itemize}
\item[A.] the    set $$\{x\in I\colon \ \rho(x) > 0\}$$ is (except for a set of zero $m$-measure) a countable  union of  elements of $\mathcal{P},$ 
\item[B.] for $m$-almost every $y \in \{x\in I\colon \ \rho(x) > 0\}$  there is $P \in \cup_k \mathcal{P}^k$ and $\epsilon > 0$ such that $y \in P$, $\rho(y)\geq \epsilon$, and $\rho(x) \geq \epsilon$ for $m$-almost every point $x$ in $P$.
\end{itemize}
 \end{theorem}
 \ \\
 
 \fcolorbox{black}{white}{
\begin{minipage}{\textwidth}
\noindent {\bf Assumption  $\Cll[A]{A9}$.} We have that $T$ is {\bf transitive}, that is, for every $P,Q\in \mathcal{P}$ there is $n\geq 0$ such that $m(P\cap T^{-n}Q)> 0$. 
\end{minipage} 
}
\ \\

 \begin{corollary}[Ergodicity]\label{ergodic} Suppose  $\Crr{A000}-\Crr{A9}$.  Then there is an unique $m$-absolutely continuous invariant  probability $\mu=\rho_0 \ dm $ for $T$. Moreover 
 \begin{itemize}
 \item[A.] The probability  $\mu$ is ergodic, 
 \item[B.]  The basin of attraction of $\mu$ is $I$ (except for a set of zero $m$-measure).
 \item[C.]  The unique function $\rho \in L^1$ that satisfies (\ref{co}) is $\rho_0$.
 \item[D.] The set E  is a cyclic group. 
 \item[E.] $1$ is a simple eigenvalue.
 \end{itemize}
\end{corollary}
\ \\

 \fcolorbox{black}{white}{
\begin{minipage}{\textwidth}
\noindent {\bf Assumption  $\Cll[A]{A10}$.} We have that $T$ is {\bf topologically mixing}, that is, for every $P,Q\in \mathcal{P}$ there is $n_0\geq 0$ such that $m(P\cap T^{-n}Q)> 0$ for every $n\geq n_0$. 
\end{minipage} 
}
\ \\

\begin{corollary}[Mixing and decay of correlations]\label{mixing} Suppose  $\Crr{A000}-\Crr{A10}$.Then 
\begin{itemize}
\item[A.] There exist $\Cll{em}\geq 0$ and $\Cll[c]{em1} \in (0,1)$ such that the following holds: If $u \in \mathcal{B}^s_{p,q}$  and $v \in L^{p'}$ then
 $$\left| \int v\circ T^k u \ dm - \int v \rho_0 \ dm \int u \ dm \right|\leq \Crr{em}\Crr{em1}^k|u|_{ \mathcal{B}^s_{p,q}}|v|_{p'}.$$
 \item[B.] $E=\{1\}.$
 \item[C.] $1$ is a simple eigenvalue.
 \end{itemize}
\end{corollary} 

\vspace{1cm}
\centerline{ \bf IV. ACTION ON $\mathcal{B}^s_{p,q}$.}
\addcontentsline{toc}{chapter}{\bf IV. ACTION ON $\mathcal{B}^s_{p,q}$.}
\vspace{1cm}

\section{Notation} We will use $C_1, C_2, \dots...$ for positive constants, $\lambda_1, \lambda_2, \dots$ for positive constants smaller than one.

\begin{table}[h]
  \centering
  \caption{List of symbols}
  \label{tab:table1}
  \begin{tabular}{cc}
    \toprule
    Symbol & Description\\
    \midrule
     $\mathcal{A}^{bs}_{s,\beta,p,q'} $ & class of $(s,\beta,p,q')$-Besov's atoms \\
     $\mathcal{P}$ & good grid  of  $I$\\
     $\mathcal{H}$ & good grid  of  $J$\\
        $\mathcal{G}_r$ & grid  of  $I_r$\\
    $\mathcal{P}^k$ & partition at the $k$-th level of $\mathcal{P}$\\
    $\mathcal{B}^s_{p,q}, \mathcal{B}^s_{p,q}(\mathcal{A}^{sz}_{s,p})$ & (s,p,q)-Banach space defined by Souza's atoms\\
    $Q, W$ & elements of grids\\
    $\{g_r\}_{r\in \Lambda}$& family of potentials\\
    $|\cdot|_p$ & norm in $L^p$, $p \in [1,\infty].$ \\
    $\Phi$ & transfer operator.  \\
    $p'$ & $1/p+ 1/p'=1$.  \\
    \bottomrule
  \end{tabular}
\end{table}

\section{Proof of  Theorem \ref{key}}

Theorem \ref{key} is our main technical result and its proof takes several steps. 
\subsection{Step 1:  Dynamical Slicing}  \label{rssec}
 
\begin{figure}
\includegraphics[scale=0.4]{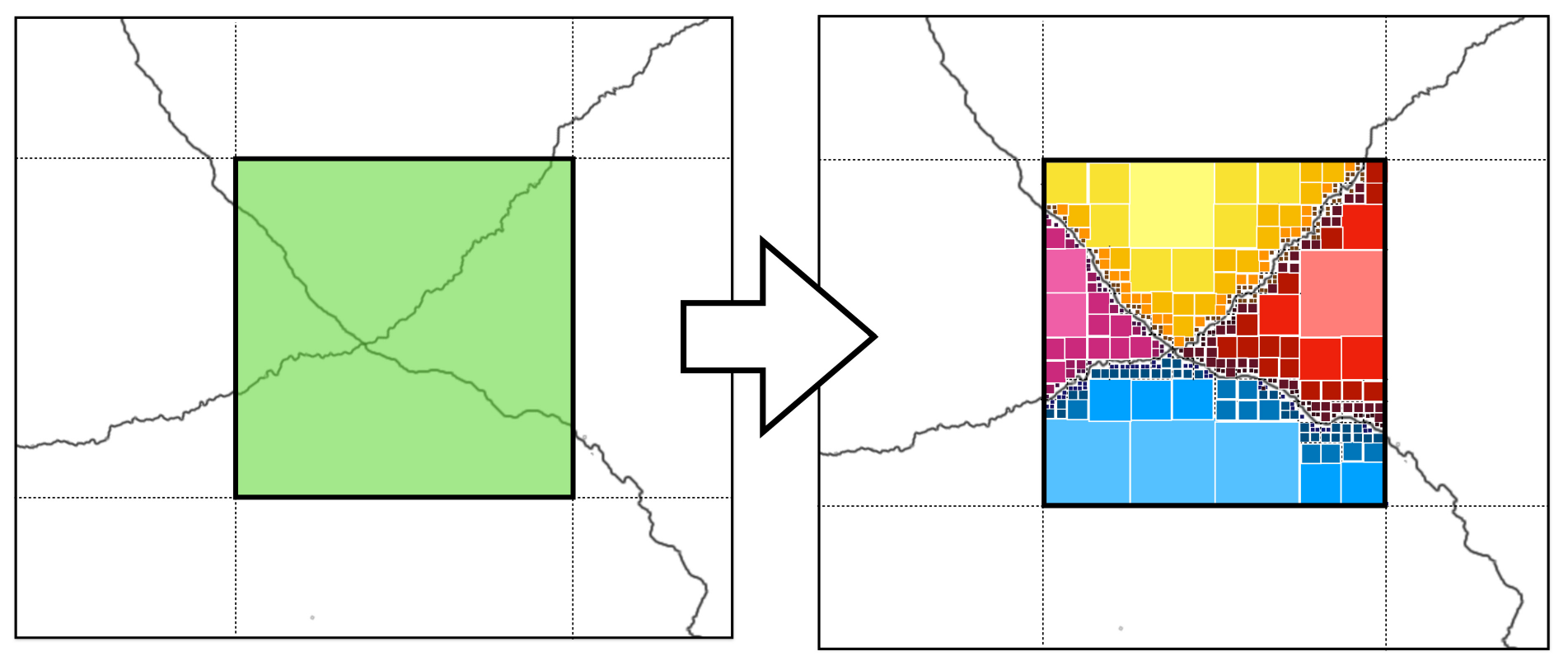}
\caption{Dynamical Slicing.  The $\mathcal{B}^s_{p,q}$ representation of the function $f$ may contains atoms $a_Q$ whose support $Q$ is not contained  in  any of the domains $I_r$. The square on the left represents the support $Q$ of an atom $a_Q$. It intercepts four domains  in the family $\{I_r\}_r$, whose boundaries are  represented by the tortuous lines.   In the dynamical slicing we replace this representation by another one where the support of each atom in it is  contained  in some domain $I_r$, as pictured on the right. }
\end{figure}

 By Section \ref{bsspq} we have that the linear application 
$$\Phi\colon \mathcal{B}^s_{p,q}(I,m,\mathcal{P}) \rightarrow L^{1}(\mu)$$
is well defined and bounded, so if   $f\in \mathcal{B}^s_{p,q}(I,m,\mathcal{P})$  satisfies  the assumptions of Theorem \ref{key} we  need to show that  $\Phi(f) \in \mathcal{B}^s_{p,q}(J,\mu,\mathcal{H})$ and estimate its norm. By assumption, for every $r \in \Lambda$  there is a $\mathcal{B}^s_{p,q}(I_r, m, \mathcal{G}_r)$-representation of  $f\cdot 1_{I_r}$
$$ f\cdot 1_{I_r}=\sum_k  \sum_{Q\in \mathcal{G}_r^k}  c_Q^r   a_Q,$$
 satisfying  either (\ref{hiip1}) or (\ref{hiip2}).

\subsection{Step 2: Applying  the transfer transformation} \ \\

\noindent {\bf Claim 1.} {\it There is  a   $\mathcal{B}^s_{p,q}(J,\mu,\mathcal{H},\mathcal{A}^{bv}_{\beta,p,q})$-representation
\begin{equation} \sum_i \sum_{W\in \mathcal{H}^i}  s_W  b_W\end{equation}
of $ \Phi(f)$  satisfying 
\begin{equation}\label{ms}   \Big( \sum_i  \big( \sum_{W\in \mathcal{H}^i}  |s_W|^p  \big)^{q/p}\Big)^{1/q} \leq \Crr{DRS1}    \big( \sum_k (\sum_{Q \in \mathcal{P}^k}    |d_Q|^p)^{q/p} \big)^{1/q}. \end{equation}}
\vspace{5mm} 

 Recall that 
$$\Phi(f)=\sum_r \Phi_r (f\cdot 1_{I_r}).$$
This series converges absolutely on $L^{1}(\mu)$, that is 
$$\sum_r |\Phi_r (f\cdot 1_{I_r})|_{1} < \infty,$$
so on $L^{1}$ we have 
$$\Phi(f)=\lim_{r_0\rightarrow\infty} \sum_{r\leq r_0} \Phi_r (f\cdot 1_{I_r}).$$
On the other hand, since $\Phi_r$ is a bounded transformation in $L^{t_0}(m)$ we have 
$$\Phi_r (f\cdot 1_{I_r}) =\sum_j   \Phi_r (\sum_{Q\in \mathcal{G}^j_r}  c_Q^r   a_Q)= \sum_j  \sum_{Q\in \mathcal{G}^j_r }  c_Q^r   \Phi_r(a_Q),  $$
So on $L^{1}$ 
$$ \sum_{r\leq r_0} \Phi_r (f\cdot 1_{I_r})=\lim_{j_0\rightarrow \infty}  \sum_{r\leq r_0}  \sum_{j\leq j_0}  \sum_{Q\in \mathcal{G}^j_r}  c_Q^r   \Phi_r(a_Q).$$
Note that all the sums on the r.h.s. are {\it finite}. To prove Claim 1. it is enough to show
\vspace{5mm} 

\begin{figure}
\includegraphics[scale=1]{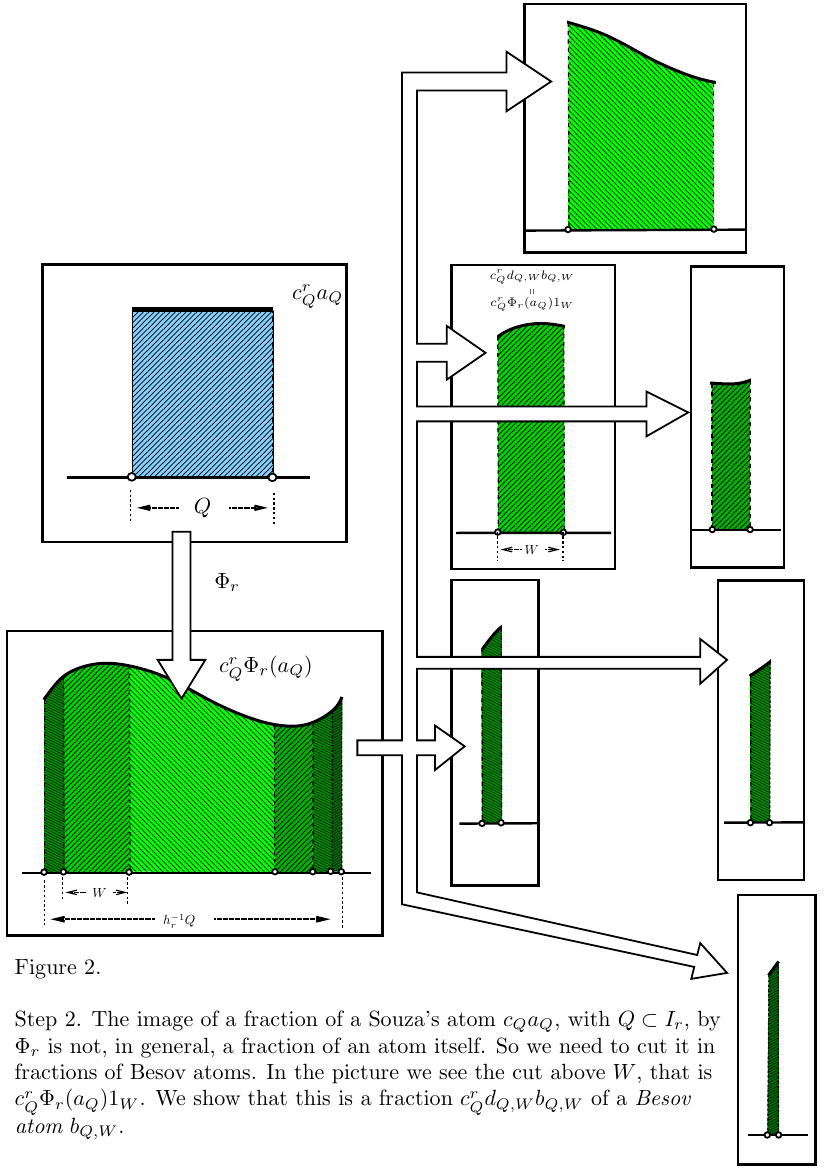}
\end{figure}

\noindent {\bf Claim 2.} {\it For every $r_0$ and $j_0$  we have that 
$$\phi_{i_0,j_0}= \sum_{r\leq r_0}  \sum_{j\leq j_0}  \sum_{Q\in \mathcal{G}^j_r}  c_Q^r   \Phi_r(a_Q)$$
has a $\mathcal{B}^s_{p,q}(J,\mu,\mathcal{H},\mathcal{A}^{bv}_{\beta,p,q})$-representation 
$$\phi_{r_0,j_0}= \sum_k\sum_{W\in \mathcal{H}^k} s_W  b_W$$
satisfying (\ref{ms}). }
\vspace{5mm}

\noindent Indeed, if  Claim 2. holds then $\Phi(f)$ belongs to the $L^1$-closure of $S=\{\phi_{r_0,j_0}\}_{r_0,j_0}$, so   Proposition \ref{compactness} implies that  Claim 1. holds.

 For every $Q \subset I_r$, with $Q\in \mathcal{G}^j_r$ we have that  $h_r^{-1}(Q)$ is a $(s,p,\Crr{DGD1},\Crr{DGD2})$-regular  domain of $(J,\mu,\mathcal{H})$, so we can consider the corresponding families   $\mathcal{F}^k(h_r^{-1}(Q))\subset \mathcal{H}^k$, $k\geq  k_0(h_r^{-1}(Q),\mathcal{H})$.  Since
$$1_{h_r^{-1}(Q)}=\sum_{k\geq k_0(h_r^{-1}(Q),\mathcal{H})}\sum_{W\in \mathcal{F}^k(h_r^{-1}(Q))} 1_W.$$
\vspace{5mm} 

\noindent {\bf Claim 3.}  {\it We have the following limit on  $L^1(\mu)$ 
$$\Phi_r(a_Q)=\lim_{K\rightarrow \infty} \sum_{\substack{ k\leq K\\k\geq k_0(h_r^{-1}(Q),\mathcal{H})}}\sum_{W\in \mathcal{F}^k(h_r^{-1}(Q))} \Phi_r(a_Q)1_W.$$}
\vspace{5mm} 

\noindent Note that this  limit holds pointwise almost everywhere since $\Phi_r(a_Q)$ vanishes outside $h_r^{-1}(Q)$. Furthermore  since $\Phi_r(a_Q)\in L^1(\mu)$ and  the elements of 
$$ \bigcup_{k\geq k_0(h_r^{-1}(Q),\mathcal{H})} \mathcal{F}^k(h_r^{-1}(Q))$$
are pairwise disjoint we have 
$$ \sum_{k\geq k_0(h_r^{-1}(Q))}\sum_{W\in \mathcal{F}^k(h_r^{-1}(Q))} |\Phi_r(a_Q)1_W|_1\leq |\Phi_r(a_Q)|_1<\infty.$$
In particular the sequence in Claim 3 converges absolutely in $L^1(\mu)$  to $\Phi_r(a_Q)$. This proves  Claim 3. So again by Proposition \ref{compactness}, we reduce the proof of Claim 2. to  the following claim
\vspace{5mm} 

\noindent {\bf Claim 4.} {\it For every $r_0$ and $j_0$  we have that 
$$\phi_{i_0,j_0,K}= \sum_{r\leq r_0}  \sum_{j\leq j_0}  \sum_{Q\in \mathcal{P}^j, Q \subset I_r }  \sum_{\substack{ k\leq K\\k\geq k_0(h_r^{-1}(Q),\mathcal{H})}}\sum_{W\in \mathcal{F}^k(h_r^{-1}(Q))} c_Q^r  \Phi_r(a_Q)1_W.$$
has a $\mathcal{B}^s_{p,q}(\mathcal{A}^{bv}_{\beta,p,q})$-representation 
$$\phi_{r_0,j_0,K}= \sum_k\sum_{W\in \mathcal{P}^k} s_W  b_W$$
satisfying (\ref{ms}). }
 \vspace{5mm} 

Note that for every pair $(Q,W)$ such that $W\subset h^{-1}_r(Q)$, $Q \subset I_r$, we have
$$\Phi_r(a_Q) 1_W =g\cdot a_Q\circ h_r 1_W = g\cdot |Q|^{s-1/p} 1_Q\circ h_r 1_W=g |Q|^{s-1/p} 1_W.$$   So by (\ref{supe33})
$$|\Phi_r(a_Q)1_W|_{\mathcal{B}^\beta_{p,q}(W,\mathcal{H}_W, \mathcal{A}^{sz}_{\beta,p}) }  \leq  \Crr{DRP}^r \Big( \frac{|Q|}{|h_r^{-1}(Q)|}\Big)^{1/p-s+\epsilon_r} |W|^{1/p-\beta} |Q|^{s-1/p}.$$
If  $W\in \mathcal{F}^k(h^{-1}_r(Q))$, $Q \subset I_r$ and  $c^r_Q\neq 0$,  define
$$b_{Q,W}(x) = \frac{1}{\Crr{DRP}^r}\Big( \frac{|h^{-1}_rQ|} {|Q| }\Big)^{1/p-s+\epsilon_r} |Q|^{1/p-s}|W|^{s-1/p}  \Phi_r(a_Q) 1_W,$$
and
\begin{equation*} d_{Q,W}= \Crr{DRP}^r\frac{|Q|^{\epsilon_r} }{|h_r^{-1}(Q)|^{\epsilon_r}} \frac{|W|^{1/p-s} }{|h_r^{-1}(Q)|^{1/p-s}}.\end{equation*}
Otherwise define $d_{Q,W}=0$ and $b_{Q,W}(x)=0$ everywhere.
\vspace{5mm} 

\noindent {\bf Claim 5.} {\it $b_{Q,W}$ is a $\mathcal{A}^{bv}_{\beta,p,q}$-atom supported  on $W$. }
\vspace{5mm} 

Indeed 
\begin{eqnarray}
|b_{Q,W}|_{\mathcal{B}^\beta_{p,q}(W,\mathcal{H}_W, \mathcal{A}^{sz}_{\beta,p}) }  
&\leq& \big( \frac{|h^{-1}_rQ|} {|Q| }\big)^{1/p-s+\epsilon_r} \big( \frac{|Q| }{|h^{-1}_rQ|} \big)^{1/p-s+\epsilon_r} |W|^{1/p-\beta}|W|^{s-1/p} \nonumber \\
&\leq&   |W|^{1/p-\beta }  |W|^{s-1/p} =   |W|^{s-\beta}\nonumber.
\end{eqnarray}
This proves Claim 5.  
We have 
$$c^r_Q \Phi_r(a_Q)1_W =  c^r_Q d_{Q,W} b_{Q,W}.$$
For every $W \in \mathcal{H}^k$ and $r\in \Lambda$ and $j\in \mathbb{N}$ there exists at most one set $Q^r_j(W) \subset I_r$ such that  $Q^r_j(W) \in \mathcal{G}^j_r$, $W \subset h_r^{-1}(Q^r_j(W))$ and $W\in \mathcal{F}^k(h_r^{-1}(Q^r_j(W))$. If such set  does not exist define $Q^r_j(W)=\emptyset$.   We have 

\begin{eqnarray*}
\phi_{i_0,j_0,K}&=& \sum_{r\leq r_0}  \sum_{j\leq j_0}  \sum_{Q\in \mathcal{P}^j, Q \subset I_r }  \sum_{\substack{ k\leq K\\k\geq k_0(h_r^{-1}(Q))}}\sum_{W\in \mathcal{F}^k(h_r^{-1}(Q))} c_Q^r  \Phi_r(a_Q)1_W\\
&=&\sum_{r\leq r_0}  \sum_{j\leq j_0}  \sum_{Q\in \mathcal{P}^j, Q \subset I_r }  \sum_{\substack{ k\leq K\\k\geq k_0(h_r^{-1}(Q))}}\sum_{W\in \mathcal{F}^k(h_r^{-1}(Q))} c^r_Q d_{Q,W} b_{Q,W}\\
&=&\sum_{r\leq r_0}  \sum_{j\leq j_0}  \sum_{\substack{W\in \mathcal{H}^k\\k\leq K}} c^r_{Q^r_j(W)} d_{Q^r_j(W),W} b_{Q^r_j(W),W} \\
&=&\sum_{\substack{W\in \mathcal{H}^k\\k\leq K}}  \sum_{r\leq r_0}  \sum_{j\leq j_0}  c^r_{Q^r_j(W)} d_{Q^r_j(W),W} b_{Q^r_j(W),W} \\
 \end{eqnarray*}
Define
$$s_W =  \sum_{r \leq r_0}\sum_{j\leq j_0}    |c^r_{Q^r_j(W)} d_{Q^r_j(W),W}| $$
Note that these sums have a  {\it finite} number of terms. Then
$$b_W= \frac{1}{s_W} \sum_{r\leq r_0}   \sum_{j\leq j_0}  c^r_{Q^r_j(W)} d_{Q^r_j(W),W} b_{Q^r_j(W),W} $$
is  a $\mathcal{A}^{bv}_{\beta,p,q}$-atom supported  on $W$, since it is a convex combination of   $\mathcal{A}^{bv}_{\beta,p,q}$-atoms. We obtained a  
$\mathcal{B}^s_{p,q}(J,\mu,\mathcal{H},\mathcal{A}^{bv}_{\beta,p,q})$-representation 
$$\phi_{i_0,j_0,K}= \sum_{k\leq K} \sum_{W\in \mathcal{H}^k}   s_W b_W.$$
Now it remains to prove (\ref{ms}). 
Recall that  $\Crr{DRS2}^r=\max \{ (\Crr{DC2}^r)^{\epsilon},(\Crr{DGD2}^r)^{1/p}\}$. 
\vspace{5mm} 

\noindent {\bf Claim 6.} {\it  If   $Q\in \mathcal{P}^j$ and $W\in \mathcal{F}^i(h^{-1}_r(Q))$ then 
\begin{equation}\label{ccoovv} (\Crr{DRS2}^r)^{|j-k_0(h_r^{-1}(Q))|}(\Crr{DRS2}^r)^{i-k_0(h_r^{-1}(Q))} \leq  (\Crr{DRS2}^r)^{\gamma|j-i| + (1-\gamma)a_r}.\end{equation}  }
\vspace{5mm} 
Indeed, note that if  $i,j, k \in \mathbb{N}$, with $|j-k|\geq a\geq 0$ and $i\geq k$ then
\begin{eqnarray*}
 i-k+|j-k|&=& |i-k|+|j-k|  \\
&= &\gamma(|i-k|+|j-k|)+ (1-\gamma)(|i-k|+|j-k|)\\
&\geq& \gamma |i-j| + (1-\gamma)a. 
\end{eqnarray*}
In particular since $$|j-k_0(h_r^{-1}(Q)|=|k_0(Q)-k_0(h_r^{-1}(Q)|\geq a_r$$ and
$W\in \mathcal{F}^i(h^{-1}_r(Q))$ implies  $i\geq k_0(h_r^{-1}(Q))$ we have (\ref{ccoovv}). This proves Claim 6. 

Since $h_r^{-1}(Q^r_j(W))$ is a $(1-sp,\Crr{DGD1},\Crr{DGD2})$-regular set by (\ref{dom})  and (\ref{tam}), if (\ref{hiip1}) holds then  for every $i\leq K$ 
\begin{eqnarray} 
 && \Big( \sum_{W\in \mathcal{H}^i}  |s_W|^p  \Big)^{1/p} \nonumber \\
 &\leq&  \Big( \sum_{W\in \mathcal{H}^i} (\sum_{r\leq r_0}   \sum_{j \leq j_0}     | c^r_{Q^r_j(W)} d_{Q^r_j(W),W}|)^p \Big)^{1/p}
\nonumber \\
&\leq& \sum_{r\leq r_0}   \sum_{j \leq j_0}   \Big( \sum_{W\in \mathcal{H}^i}  | c^r_{Q^r_j(W)} d_{Q^r_j(W),W}|^p \Big)^{1/p}  
\nonumber \\
&\leq& \sum_{r\leq r_0}   \sum_{j \leq j_0}   \Big( \sum_{W\in \mathcal{H}^i} (\Crr{DRP}^r)^p \Big( \frac{|Q^r_j(W)| }{|h_r^{-1}(Q^r_j(W))|} \Big)^{\epsilon_r p}\Big( \frac{|W|}{|h_r^{-1}(Q^r_j(W))|}\Big)^{1-sp} |c_{Q^r_j(W)}^r|^p \Big)^{1/p}  
\nonumber \\
&\leq&  \sum_{r\leq r_0}    \sum_{j \leq j_0}  (\Crr{DC1}^r)^{\epsilon} \Crr{DRP}^r   \Big( \sum_{W\in \mathcal{H}^i}  (\Crr{DC2}^r)^{\epsilon p |j-k_0(h_r^{-1}(Q^r_j(W)))|}\Big( \frac{|W|}{|h_r^{-1}(Q^r_j(W))|}\Big)^{1-sp} |c_{Q^r_j(W)}^r|^p \Big)^{1/p}  
\nonumber \\
&\leq&   \sum_{r\leq r_0}   \sum_{j \leq j_0}  (\Crr{DC1}^r)^{\epsilon} \Crr{DRP}^r  \Big(  \sum_{Q\in \mathcal{G}^j_r }    (\Crr{DC2}^r)^{\epsilon p |j-k_0(h_r^{-1}(Q))|}   |c_{Q}^r|^p   \sum_{\substack{ W\in \mathcal{H}^i \\ Q^r_j(W)=Q \\ i\geq k_0(h_r^{-1}(Q))} }   \Big(\frac{|W|}{|h_r^{-1}(Q)|}\Big)^{1-sp}  \Big)^{1/p}  
\nonumber \\
&\leq&     \sum_{r\leq r_0}  \sum_{j \leq j_0} (\Crr{DC1}^r)^{\epsilon} \Crr{DRP}^r   \Big(   \sum_{\substack{Q\in \mathcal{G}^j_r \\ i\geq k_0(h_r^{-1}(Q))}}  (\Crr{DC2}^r)^{\epsilon p |j-k_0(h_r^{-1}(Q))|}  \Crr{DGD1}^r (\Crr{DGD2}^r)^{i-k_0(h_r^{-1}(Q))}  |c_{Q}^r|^p  \Big)^{1/p}  
\nonumber \\
&\leq&  \sum_{j \leq j_0}    \Crr{DRS2}^{\gamma |j-i|}   \sum_{r\leq r_0}  (\Crr{DC1}^r)^{\epsilon} \Crr{DRP}^r  ( \Crr{DGD1}^r)^{1/p} ( \Crr{DRS2}^r)^{(1-\gamma)a_r}  \big(   \sum_{Q\in \mathcal{G}^j_r}  |c_Q^r|^p \big)^{1/p}. \nonumber 
 \end{eqnarray}
 
 This is a convolution, so for $q\in [1,\infty)$
\begin{align*} 
&\Big( \sum_i  \big( \sum_{W\in \mathcal{H}^i}  |s_W|^p  \big)^{q/p}\Big)^{1/q} \nonumber \\ 
&\leq   \frac{2}{1-\Crr{DRS2}^{\gamma}}     \Big( \sum_j \big( \sum_r    (\Crr{DC1}^r)^{\epsilon} \Crr{DRP}^r  ( \Crr{DGD1}^r)^{1/p}  (\Crr{DRS2}^r)^{(1-\gamma)a_r} \big( \sum_{Q\in \mathcal{G}^j_r}  |c_Q^r|^p \big)^{1/p}\big)^{q} \Big)^{1/q} \nonumber \\
&\leq \frac{2 \Crr{DRS1}}{1-\Crr{DRS2}^{\gamma}}   \big( \sum_k (\sum_{Q \in \mathcal{P}^k}    |d_Q|^p)^{q/p} \big)^{1/q}.  
\end{align*}

 The case $q=\infty$ is similar. On the other hand,  if (\ref{hiip2}) holds then  for every $i\leq K$ 
\begin{eqnarray} 
 && \Big( \sum_{W\in \mathcal{H}^i}  |s_W|^p  \Big)^{1/p} \nonumber \\
 &\leq&   \Big( \sum_{W\in \mathcal{H}^i} (\sum_{r\leq r_0}   \sum_{j \leq j_0}     | c^r_{Q^r_j(W)} d_{Q^r_j(W),W}|)^p \Big)^{1/p}
\nonumber \\
&\leq&    \sum_{j \leq j_0}   \Big( \sum_{W\in \mathcal{H}^i}   (  \sum_{r\leq r_0}  |c^r_{Q^r_j(W)} d_{Q^r_j(W),W}| )^p \Big)^{1/p}  
\nonumber \\
&\leq&  N^{1/p'}   \sum_{j \leq j_0}     \Big( \sum_{W\in \mathcal{H}^i}   \sum_{r\leq r_0}  ( \Crr{DRP}^r)^p \Big( \frac{|Q^r_j(W)| }{|h_r^{-1}(Q^r_j(W))|} \Big)^{\epsilon_r p}\Big( \frac{|W|}{|h_r^{-1}(Q^r_j(W))|}\Big)^{1-sp} |c_{Q^r_j(W)}^r|^p \Big)^{1/p}  
\nonumber \\
&\leq&  N^{1/p'}  \sum_{j \leq j_0}    \Big( \sum_{W\in \mathcal{H}^i}   \sum_{r\leq r_0}  (\Crr{DC1}^r)^{\epsilon p} (\Crr{DRP}^r)^p   \Crr{DC2}^{\epsilon p |j-k_0(h_r^{-1}(Q^r_j(W)))|}\Big( \frac{|W|}{|h_r^{-1}(Q^r_j(W))|}\Big)^{1-sp} |c_{Q^r_j(W)}^r|^p \Big)^{1/p}  
\nonumber \\
&\leq& N^{1/p'}   \sum_{j \leq j_0}  \Big(    \sum_{r\leq r_0}      \sum_{Q\in \mathcal{G}^j_r }  (\Crr{DC1}^r)^{\epsilon p}( \Crr{DRP}^r)^p  (\Crr{DC2}^r)^{\epsilon p |j-k_0(h_r^{-1}(Q))|}   |c_{Q}^r|^p   \sum_{\substack{ W\in \mathcal{H}^i \\ Q^r_j(W)=Q \\ i\geq k_0(h_r^{-1}(Q))} }   \Big(\frac{|W|}{|h_r^{-1}(Q)|}\Big)^{1-sp}  \Big)^{1/p}  
\nonumber \\
&\leq&  N^{1/p'} \sum_{j \leq j_0}  \Big(  \sum_{r\leq r_0} ( \Crr{DC1}^r)^{\epsilon p} ( \Crr{DRP}^r)^p  \sum_{\substack{ Q\in \mathcal{G}^j_r\\ i\geq k_0(h_r^{-1}(Q))}}  (\Crr{DC2}^r)^{\epsilon p |j-k_0(h_r^{-1}(Q))|}  \Crr{DGD1}^r (\Crr{DGD2}^r)^{i-k_0(h_r^{-1}(Q))}  |c_{Q}^r|^p  \Big)^{1/p}  
\nonumber \\
&\leq& N^{1/p'}  \sum_{j \leq j_0} \Crr{DRS2}^{\gamma |j-i|}     \big(    \sum_{r\leq r_0}   ( \Crr{DC1}^r)^{\epsilon p} ( \Crr{DRP}^r)^p  \Crr{DGD1}^r  (\Crr{DRS2}^r)^{(1-\gamma)pa_r}    \sum_{Q\in \mathcal{G}^j_r}  |c_Q^r|^p \big)^{1/p}. \nonumber 
 \end{eqnarray}
 
  This is again a  convolution, so  for $q\in [1,\infty)$ we have 
\begin{align*} 
&\Big( \sum_i  \big( \sum_{W\in \mathcal{H}^i}  |s_W|^p  \big)^{q/p}\Big)^{1/q} \nonumber \\ 
&\leq   N^{1/p'}  \frac{2}{1-\Crr{DRS2}^{\gamma}}     \Big( \sum_j \big( \sum_r  ( ( \Crr{DC1}^r)^{\epsilon} \Crr{DRP}^r ( \Crr{DGD1}^r )^{1/p} ( \Crr{DRS2}^r)^{(1-\gamma)a_r})^p \sum_{Q\in \mathcal{G}^j_r}  |c_Q^r|^p \big)^{q/p} \Big)^{1/q} \nonumber \\
&\leq \frac{2 \Crr{DRS1}}{1-\Crr{DRS2}^{\gamma}}   \big( \sum_k (\sum_{Q \in \mathcal{P}^k}    |d_Q|^p)^{q/p} \big)^{1/q}.  
\end{align*}

and the  case $q=\infty$ is similar.

\subsection{Step 3.  Going back to $\mathcal{B}^s_{p,q}$} \label{rssec2}
 By Proposition \ref{besova} there is a $\mathcal{B}^s_{p,q}$-representation  
$$\Phi(f)= \sum_k \sum_{Q\in \mathcal{H}^k} z_Q  a_Q$$
such that 
\begin{align*}
&\Big(  \sum_i \Big(  \sum_{Q\in \mathcal{H}^i} |z_Q|^p    \Big)^{q/p}  \Big)^{1/q} \\
&\leq \Crr{GBS} \Big( \sum_i  \big( \sum_{W\in \mathcal{H}^i}  |s_W|^p  \big)^{q/p}\Big)^{1/q} \\
&\leq \frac{2 \Crr{GBS}  \Crr{DRS1}}{1-\Crr{DRS2}^{\gamma}}   \big( \sum_k (\sum_{Q \in \mathcal{P}^k}    |d_Q|^p)^{q/p} \big)^{1/q} \end{align*}

This completes the proof of Theorem \ref{key}.

\section{Controlling the Essential Spectral Radius}
\begin{proof}[Proof of Corollary \ref{ww}]  Given $\mathcal{R}\subset \mathcal{P}$,  Proposition \ref{canrep2} tell us that 
$$\pi_{\mathcal{G}}(f)=\sum_k \sum_{P\in \mathcal{P}^k, P\in \mathcal{R}} k_P^f a_P$$
is a $\mathcal{B}^s_{p,q}(I,m,\mathcal{P})$-representation of $\pi_{\mathcal{R}}(f)$ satisfying
$$\Big( \sum_k \big( \sum_{P\in \mathcal{P}^k, P\in \mathcal{R}} |k_P^f|^p \big)^{q/p} \Big)^{1/q}\leq \Crr{GC} |f|_{\mathcal{B}^s_{p,q}}.$$
If the pair 
$$(\mathcal{I},\sum_k \sum_{P\in \mathcal{P}^k, P\in \mathcal{R}} k_P^f a_P ),$$
has a  $(C,\gamma)$-regular slicing then by Theorem \ref{key} we have  that
$$|\Phi(\pi_{\mathcal{R}}(f))|_{\mathcal{B}^s_{p,q}}\leq  \Crr{GBS}\Crr{D}  C \Crr{GC} |f|_{\mathcal{B}^s_{p,q}}.$$
Applying this inequality for $\mathcal{R}=\mathcal{P}'$ and $\mathcal{R}=\mathcal{P}\setminus \mathcal{P}'$ we conclude the proof.
\end{proof}

\vspace{1cm}
\centerline{ \bf V. POSITIVE TRANSFER OPERATORS.}
\addcontentsline{toc}{chapter}{\bf V. POSITIVE TRANSFER OPERATORS.}
\vspace{1cm}

\section{Lasota-Yorke Inequality and the dynamics of $\Phi$ } 

\begin{proof}[Proof of Theorem \ref{t1}] Note that 
$$\pi_{\mathcal{P}'}(f)=\sum_k \sum_{P\in \mathcal{P}^k, P\in \mathcal{P}'} k_P^f a_P$$
is a $\mathcal{B}^s_{p,q}$-representation of $\pi_{\mathcal{P}'}(f)$. Moreover $f\mapsto k^f_P$ is a bounded linear functional in $(L^1)^\star$.  Denote its norm by $|k_P|_{(L^1)^\star}$. So since $\mathcal{P}'$ is finite 
$$\Big( \sum_k \big( \sum_{P\in \mathcal{P}^k, P\in \mathcal{P}'} |k_P^f|^p \big)^{q/p} \Big)^{1/q}\leq  \Big( \sum_k \big( \sum_{P\in \mathcal{P}^k, P\in \mathcal{P}'} |k_P^f|^p_{(L^1)^\star} \big)^{q/p} \Big)^{1/q} |f|_1.$$
The pair 
$$(\mathcal{I},\sum_k \sum_{P\in \mathcal{P}^k, P\in \mathcal{G}} k_P^f a_P ),$$
has a  $(\Crr{DRSFR},\gamma)$-regular slicing so by Theorem \ref{key} we have  that
$$|\Phi(\pi_{\mathcal{P}'}(f))|_{\mathcal{B}^s_{p,q}}\leq  \Crr{GBS}\Crr{D}   \Crr{DRSFR}  |f|_1.$$
Consequently Corollary \ref{ww} gives
\begin{eqnarray*}
|\Phi(f)|_{\mathcal{B}^s_{p,q}}&\leq&   |\Phi(\pi_{\mathcal{P}'}(f))|_{\mathcal{B}^s_{p,q}} +  |\Phi(\pi_{\mathcal{P}\setminus \mathcal{P}'}(f))|_{\mathcal{B}^s_{p,q}} \\
&\leq &  \Crr{GBS}\Crr{D} \Crr{DRSFR}  |f|_1  +     \Crr{GBS} \Crr{D} \Crr{DRSES}\Crr{GC}  |f|_{B^s_{p,q}}.
\end{eqnarray*} 
Using that $|\Phi(f)|_1\leq |f|_1$  and $\Crr{GBS} \Crr{D}   \Crr{DRSES} < 1$ one can easily get the Lasota-Yorke inequality for $\Phi^n$.
\end{proof}

\begin{proof}[Proof of Corollary  \ref{spec}] The methods we are going to use here are  standard, however, we provided them for the sake of compactness. \\

\noindent {\it Proof of A.} Since the  essential spectral radius of $\Phi$ is at most  $$\Crr{GBS}\Crr{D}  \Crr{DRSES} < 1$$ we have that every $$\lambda \in \sigma_{\mathcal{B}^s_{p,q}}(\Phi)$$ satisfying $ |\lambda|\geq ( \Crr{GBS} \Crr{D} \Crr{DRSES})^{1/2}$ is an isolated point of the spectrum that  is an eigenvalue with finite-dimensional generalized eigenspace. We claim that the spectral radius $r_{\mathcal{B}^s_{p,q}}$ of $\Phi$ is $1$.  Note that Lasota-Yorke inequality implies 
\begin{equation} \label{nopol} \sup_n |\Phi^n|_{\mathcal{B}^s_{p,q}}< \infty
\end{equation} 
 so $r_{\mathcal{B}^s_{p,q}}\leq 1$. On the other hand if $r_{\mathcal{B}^s_{p,q}} <  1$ then 
$$\lim_{n} \int |\Phi^n(f)| \  dm =0$$
for every $f\in \mathcal{B}^s_{p,q}$, but this is impossible since 
$$\int |\Phi^n(1)| \  dm=1$$
for every $n$.   Moreover if $(\Phi -\lambda I)^2f=0$ but $(\Phi -\lambda I)f\neq0$, with $|\lambda|=1$, then    $|\Phi^n f|_{B^s_{p,q}}$ diverges to infinity. This it is impossible.  In particular  if $\delta \in (0,1)$ is close enough to $1$ we have that $E$ is finite, non-empty and contained in $\mathbb{S}^1$. \\

\noindent {\it Proof of B.} This follows easily from $A$. using  arguments with spectral projections, since  the spectral projections on  the generalized eigenspace of  $\lambda \in E$ is indeed  a projection on the  eigenspace of $\lambda$.\\

\noindent {\it Proof of C.}  Let $f\geq 0$ with $f\in \mathcal{B}^s_{p,q}$. The Lasota-Yorke inequality implies that there is $N$ such that 
$$|\Phi^n(f)|_{\mathcal{B}^s_{p,q}}\leq 2C|f|_1$$
for every $n\geq N$. The ball of center $0$ and radius $2C$ in $\mathcal{B}^s_{p,q}$ is compact in $L^1$, so we can find a convergent subsequence in $L^1$
$$\rho= \lim_k   \frac{1}{n_k} \sum_{n<  n_k}  \Phi^n(f),$$
with $\rho \in \mathcal{B}^s_{p,q}$. Of course the positivity of $\Phi$ implies $\rho\geq 0$. Note also that $\Phi(\rho)=\rho$ and $\int \rho \ dm=\int f \ dm$. We conclude the proof choosing $f=1$. \\

\noindent {\it Proof of D.} Let $u\in L^1$ such that $\Phi(u)=\lambda u,$ with $|\lambda|=1$. Since $B^s_{p,q}$ is dense in $L^1$ one can choose $u_i \in \mathcal{B}^s_{p,q}$ such that $u_i$ converges to $u$ in the topology of  $L^1$.  Due the Lasota-Yorke inequality for every large $i$ there is $n_i\geq i$ such that 
$$|\Phi^{n}(u_i)|_{\mathcal{B}^s_{p,q}}\leq 2C |u|_1$$
for $n\geq n_i$.  Here $C$ depends only on $\Phi$. Note also that 
$$\lim_i  |\Phi^{n_i}(\frac{1}{\lambda^{n_i}}u_i)- u|_1=0.$$
The ball of center $0$ and radius $2C |u|_1$ in $\mathcal{B}^s_{p,q}$ is compact in $L^1$, so $u\in \mathcal{B}^s_{p,q}$ and $|u|_{\mathcal{B}^s_{p,q}}\leq 2C|u|_1$.  \\

\noindent Finally note that $\Phi(|u|)(x)\geq |u|(x)$ almost everywhere. On the other hand $$\int \Phi(|u|) \ dm = \int |u| \ dm,$$
so $\Phi(|u|)= |u|$ almost everywhere. Denote
$$s(x) = (sgn \ u)(x)\begin{cases} \frac{u(x)}{|u(x)|}      & if  \ u(x)\neq 0  \\
                                                     0                          & if  \ u(x)=0.\end{cases}$$
Note that 
 $$\sum_r  g_r(x)s(h_r(x))|u|(h_r(x))=\lambda s(x) |u|(x).$$
 and
  $$\sum_r  g_r(x)|u|(h_r(x))= |u|(x).$$
  So
  $$\big|  \sum_r  g_r(x)s(h_r(x))|u|(h_r(x))   \big|=  \sum_r  g_r(x)|u|(h_r(x)).$$
 for $m$-almost everywhere $x$,  which implies that for $m$-almost everywhere $x$ we have  $s(h_r(x))=s(h_{r'}(x))$ for every $r,r'$ such that $x\in J_r\cap J_{r'}$, $g_r(x)g_{r'}(x)\neq 0$ , $s(h_r(x))s(h_{r'}(x))\neq 0$ and consequently $s(h_r(x))=\lambda s(x)$ for $x\in J_r$ satisfying $g_r(x)\neq 0$ and $s(h_r(x))\neq 0$.  In particular $s^k(h_r(x))=\lambda^k s^k(x)$ under the same conditions, with $k\in \mathbb{Z}$ (here we define $s^k(x)=0$ whenever $s(x)=0$), and  it is easy to see that $\Phi(s^k|u|)=\lambda^k s^k|u|.$  Of course  $u_k=s^k|u| \in L^1$, so $u_k\in \mathcal{B}^s_{p,q}$. \\
  
  \noindent {\it Proof of E.}  If $\lambda \in E$ is not a root of unit then $\lambda^k \neq \lambda^{k'}$ for $k\neq k'$. So \{$\lambda^k\}_{k\in \mathbb{N}}$ is an infinite set. But by D. this set is contained in $E$, that is finite. This is a contradiction.  
\end{proof}

\begin{lemma} \label{erg} Let $\mu$ be a finite invariant measure of $T$ such that $\mu$ is absolutely continous with respect to $m$. Let $\Omega_\mu =\{ x\colon \  \rho(x)> 0\}$, where $\rho$ is the density of $\mu$ with respect to $m$. Then there is  an ergodic probability measure $\hat{\mu}$, absolutely continuous with respect to $m$, such that $\Omega_{\hat{\mu}}\subset \Omega_\mu$.
\end{lemma}
\begin{proof} Suppose that such $\hat{\mu}$ does not exist.  Then  it is easy to construct an infinite sequence of subsets $\Omega_\mu=A_0 \supset A_1 \supset  A_2 \supset \cdots $ such that $m(A_{i +1})< m(A_i)$ and $\mu_i(S)=\mu(S\cap ( A_i\setminus A_{i+1}))$ is a (no vanishing)  finite invariant measure for $T$.  Note that $\Omega_{\mu_i}\cap \Omega_{\mu_{i+1}}=0$, so if $\mu_i= \rho_i \ m$ then by Corollary  \ref{spec}.D we have that $\{\rho_i\}_i$ is  linearly independent family of functions in $\mathcal{B}^s_{p,q}$ satisfying  $\Phi(\rho_i)=\rho_i$, so the $1$-eigenspace has infinite dimension. This contradicts Corollary  \ref{spec}.A.\end{proof}

\begin{proof}[Proof of Corollary \ref{phy}] If $\mu_i=\rho_i \ m$,with $\rho_i \geq 0$,  $i=1,\dots, n$,  are distinct ergodic invariant probabilities of $T$ then by Corollary  \ref{spec}.D we have that $\rho_i \in \mathcal{B}^s_{p,q}$. Since they are ergodic and distinct we have that the sets $\Omega_i =\{ x \in I\colon \ \rho_i(x)> 0 \}$ are pairwise disjoint, so  $\rho_i$ are linearly independent on $\mathcal{B}^s_{p,q}$. Since these functions belong to the $1$-eigenspace of $\Phi$ we have that $n$ is bounded by the (finite) dimension of this eigenspace. Form now one let  $\mu_i$, $i=1,\dots, n$ be the list of all distinct ergodic invariant probabilities of $T$ absolutely continuous with respect to $m$.  We claim that 
$$\Omega^c =   \big(\cup_i  \cup_{j\geq 0}T^{-j}(\Omega_i)\big)^c$$
satisfies $m(\Omega^c)=0$. Indeed, otherwise consider $u=1_{\Omega^c}$.    Now we use an argument similar to the proof of Corollary  \ref{spec}.D. Since $B^s_{p,q}$ is dense in $L^1$ one can choose $u_i \in \mathcal{B}^s_{p,q}$ such that 
$$|u_i-u|_1\leq \frac{1}{i}$$
Due the Lasota-Yorke inequality for every large $i$ there is $n_i\geq i$ such that 
$$|\Phi^{n}(u_i)|_{\mathcal{B}^s_{p,q}}\leq 2C |u|_1$$
for $n\geq n_i$, so there is $N_i$ such that for $N\geq N_i$ we have
$$|\frac{1}{N}\sum_{n\leq N} \Phi^{n}(u_i)|_{\mathcal{B}^s_{p,q}}\leq 3C |u|_1$$
  Here $C$ depends only on $\Phi$. The ball of center $0$ and radius $3C |u|_1$ in $\mathcal{B}^s_{p,q}$ is compact in $L^1$, so we can use the Cantor diagonal argument to show that there is a sequence $M_k$ and $v_i\in \mathcal{B}^s_{p,q}$ such that  $|v_i|_{\mathcal{B}^s_{p,q}}\leq 3C|u|_1$  satisfying
$$\lim_k |\frac{1}{M_k}\sum_{n\leq M_k} \Phi^{n}(u_i)-v_i|_1=0$$
for every $i$, 
and $\Phi(v_i)=v_i$, $v_i\geq 0$. Using a similar argument we can assume that $\lim_i v_i=v$ on $L^1$, with $v\in \mathcal{B}^s_{p,q}$. Note also that 
$$|\frac{1}{M_k}\sum_{n\leq M_k} \Phi^{n}(u_i)-\frac{1}{M_k}\sum_{n\leq M_k} \Phi^{n}(u)|_{L^1}\leq \frac{1}{i},$$
so we conclude that 
$$\lim_i |\frac{1}{M_k}\sum_{n\leq M_k} \Phi^{n}(u)-v|_1=0.$$
Since $T^{-1}(\Omega^c)\subset \Omega^c$ we have that  $\Phi^j(u)$ vanishes  outside $\Omega^c$ for every $j$, so $v=0$ outside $\Omega^c$ but 
$$\int v \ dm = m(\Omega^c) > 0.$$
Note that since $\mu=v m$ is an invariant finite measure  then by Lemma \ref{erg} there is ergodic probability measure $\hat{\mu}$, absolutely continuous with respect to $m$ such that $\Omega_{\hat{\mu}}\subset \Omega^c$, which contradicts the definition of $\Omega^c$. 
\end{proof}
 
 \section{Positivity, structure of  invariant measures and decay of correlations} 
 \begin{proposition} Assume $\Crr{A000}-\Crr{A8}$. Then there is $\Cll{frr}$ that depends only on $\mathcal{P}'$ such that the following holds.  Suppose that  $f \in \mathcal{B}^s_{p,q}$ has  a $\mathcal{B}^s_{p,q}$-representation
 $$f= \sum_{Q\in \mathcal{P}}d_Q^0  a_Q$$
 such that $d_Q^0 \geq 0$ for every $Q\in \mathcal{P}$. Then $\Phi^i(f)$ has a $\mathcal{B}^s_{p,q}$-representation
 $$\sum_{Q\in \mathcal{P}}d_Q^i  a_Q$$
 with $d_Q^i\geq 0$ satisfying 
\begin{align*}&\Big( \sum_k \big( \sum_{Q\in \mathcal{P}^k}|d_Q^{i}|^p \big)^{q/p} \Big)^{1/q} \nonumber\\
&\leq  \frac{ \Crr{GBS}  \Crr{D}\Crr{DRSFR}\Crr{frr} }{1- \Crr{GBS}  \Crr{D}\Crr{DRSES} } |f|_1  + (\Crr{GBS}  \Crr{D}\Crr{DRSES} )^i \Big( \sum_k \big( \sum_{Q\in \mathcal{P}^k}|d_Q^{0}|^p \big)^{q/p} \Big)^{1/q}.\end{align*}
 \end{proposition}
 \begin{proof} This proof is similar to the proof of Theorem \ref{t1}.  Consider the function $f_1\in \mathcal{B}^s_{p,q}$ given by   $\mathcal{B}^s_{p,q}$-representation 
 
$$f_1=\sum_k \sum_{P\in \mathcal{P}^k, P\in \mathcal{P}'} d_P^0 a_P.$$
Note that for $P \in \mathcal{P}'$ 
$$|d_P^0|= |P|^{s-1/p+1}\int d_P^0 a_P \ dm \leq  |P|^{s-1/p+1} \int f_1 \ dm =  |P|^{s-1/p+1} |f_1|_1.$$
So since $\mathcal{P}'$ is finite  there is $\Crr{frr}$, that depends only on $\mathcal{P}'$, such that 
$$\Big( \sum_k \big( \sum_{P\in \mathcal{P}^k, P\in \mathcal{P}'} |d_P^0|^p \big)^{q/p} \Big)^{1/q}\leq   \Crr{frr} |f_1|_1.$$
The pair 
$$(\mathcal{I},\sum_k \sum_{P\in \mathcal{P}^k, P\in \mathcal{P}'} d_P^0 a_P ),$$
has a  $(\Crr{DRSFR},\gamma)$-regular slicing so by Theorem \ref{key} there is a $\mathcal{B}^s_{p,q}$-representation
$$\Phi(g)= \sum_k \sum_{Q\in \mathcal{P}}   d_Q'' a_Q$$
with $d_Q''\geq 0$ for every $Q$, such that 
$$ \Big( \sum_k \big( \sum_{Q\in \mathcal{P}^k}|d_Q''|^p \big)^{q/p} \Big)^{1/q}\leq \Crr{GBS}  \Crr{D}\Crr{DRSFR}\Crr{frr} |f_1|_1. $$
Moreover consider $f_2\in \mathcal{B}^s_{p,q}$ with $\mathcal{B}^s_{p,q}$-representation
$$f_2=\sum_k \sum_{P\in \mathcal{P}^k, P\in \mathcal{P}\setminus \mathcal{P}'} d_P^0 a_P.$$
The pair 
$$(\mathcal{I},\sum_k \sum_{P\in \mathcal{P}^k, P\in \mathcal{P}\setminus \mathcal{P}'} d_P^0 a_P ),$$
has a  $(\Crr{DRSES},\gamma)$-regular slicing so by Theorem \ref{key} there is a $\mathcal{B}^s_{p,q}$-representation
$$\Phi(f_2)= \sum_k \sum_{Q\in \mathcal{P}}   d_Q''' a_Q$$
with $d_Q'''\geq 0$ for every $Q$, such that 
\begin{align*} 
 \Big( \sum_k \big( \sum_{Q\in \mathcal{P}^k}|d_Q'''|^p \big)^{q/p} \Big)^{1/q}&\leq \Crr{GBS}  \Crr{D}\Crr{DRSES} \Big( \sum_k \big( \sum_{P\in \mathcal{P}^k, P\in \mathcal{P}\setminus \mathcal{P}'} |d_P^0|^p \big)^{q/p} \Big)^{1/q} \\
 &\leq \Crr{GBS}  \Crr{D}\Crr{DRSES} \Big( \sum_k \big( \sum_{P\in \mathcal{P}^k} |d_P^0|^p \big)^{q/p} \Big)^{1/q}
 \end{align*} 
 Then $\Phi(f)$ as a $\mathcal{B}^s_{p,q}$-representation 
 $$\Phi(f)= \sum_k \sum_{Q\in \mathcal{P}^k} d^1_Qa_Q$$
where $d^1_Q= d_Q''+ d_Q'''$ satisfy
\begin{align*} &\Big( \sum_k \big( \sum_{Q\in \mathcal{P}^k}|d_Q^{1}|^p \big)^{q/p} \Big)^{1/q} \nonumber\\
&\leq   \Crr{GBS}  \Crr{D}\Crr{DRSFR}\Crr{frr}  |f|_1  +  \Crr{GBS}  \Crr{D}\Crr{DRSES}  \Big( \sum_k \big( \sum_{Q\in \mathcal{P}^k}|d_Q^{0}|^p \big)^{q/p} \Big)^{1/q}.\end{align*}
The conclusion of the proposition easily follows by  an induction argument on $i$ with  the above inequality and that  fact that $|\Phi(f)|_1\leq |f|_1$ for every $f\in L^1$. 
 \end{proof} 
 
 \begin{corollary}\label{4567} Suppose that  $f \in \mathcal{B}^s_{p,q}$ has  a $\mathcal{B}^s_{p,q}$-representation
 $$f= \sum_{Q\in \mathcal{P}}d_Q^0  a_Q$$
 such that $d_Q^0 \geq 0$ for every $Q\in \mathcal{P}$. Then the set
 $$\{   \frac{1}{n+1}  \sum_{i=0}^{n}\Phi^i(f)    \}_{n \in \mathbb{N}} $$
 is pre-compact in  $L^p$ and every accumulation point $\rho$ of this sequence belongs to $\mathcal{B}^s_{p,q}$ and it has a $\mathcal{B}^s_{p,q}$-representation 
 $$\rho= \sum_{Q\in \mathcal{P}}d_Q^\infty  a_Q$$
 satisfying $d_Q^\infty \geq 0$ for every $Q\in \mathcal{P}$ and
 $$\Big( \sum_k \big( \sum_{Q\in \mathcal{P}^k}|d_Q^{\infty}|^p \big)^{q/p} \Big)^{1/q}\leq  \frac{ \Crr{GBS}  \Crr{D}\Crr{DRSFR}\Crr{frr} }{1- \Crr{GBS}  \Crr{D}\Crr{DRSES} }|f|_1.  $$
  Moreover
 $$\int \rho \ dm =  \int f \ dm \ and \ \Phi(\rho)=\rho.$$
 \end{corollary} 
 \begin{proof} The proof is quite similar to the proof of Corollary  \ref{spec}.C.
 \end{proof}
 

\begin{proof}[Proof of Theorem \ref{acim}] Let $\rho \in L^1$, $\rho\geq 0$ that satisfies (\ref{co}).  There exists $f_n \in L^1$ such that $$f_k= \sum_{i\leq i_k} c_i^k 1_{S^k_i},$$ where $c^i_k\geq 0$ and $S^k_i \mathbb{A}$, and $\lim_k |\rho-f_k|_1=0$. Due G7. the algebra generated by $\cup_k \mathcal{P}^k$ generated the $\sigma$-algebra $\mathbb{A}$. In particular for every $\epsilon > 0$ and $S\in \mathbb{A}$ there is $\hat{S}$ in the algebra generated by $\cup_k \mathcal{P}^k$ such that $m(S\triangle \hat{S}) < \epsilon$. Consequently there is $g_k$ such that 
$$g_k= \sum_{i\leq j_k} c_i^k 1_{W^k_i},$$ where $q^i_k\geq 0$ and $W^k_i \in \cup_k \mathcal{P}^k$, satisfying 
$\lim_n |f_k-g_k|_1=0$, so  $\lim_n |\rho-g_k|_1=0$. and we may assume  $|g_k|_1\leq 2|\rho|_1+1$. Note that 
 $$\Big| \frac{1}{n+1}  \sum_{i=0}^{n}\Phi^i(g_k)  -  \rho\Big|_1 \leq |g_k-\rho|_1.$$
By Corollary \ref{4567} there are $\rho_k$ such that
$$\rho_k= \sum_{Q\in \mathcal{P}}d_Q^k  a_Q$$
 satisfying $d_Q^k \geq 0$ for every $Q\in \mathcal{P}$, and
\begin{eqnarray}\label{dccc}
\Big( \sum_j \big( \sum_{Q\in \mathcal{P}^j}|d_Q^{k}|^p \big)^{q/p} \Big)^{1/q}\leq  \frac{ \Crr{GBS}  \Crr{D}\Crr{DRSFR}\Crr{frr} }{1- \Crr{GBS}  \Crr{D}\Crr{DRSES} }( 2|\rho|_1+1),\end{eqnarray} 
and moreover $\Phi(\rho_k)=\rho_k.$ and 
  $$|\rho_k -  \rho|_1 \leq |g_k-\rho|_1.$$
In particular $\lim_k |\rho_k - \rho |_1 =0$. Due (\ref{dccc}) and Proposition \ref{compactness}.B one can find a representation
$$\rho= \sum_{Q\in \mathcal{P}}d_Q^\infty  a_Q$$
 satisfying $d_Q^\infty \geq 0$ for every $Q\in \mathcal{P}$, and such that (\ref{dccc}) holds for $k=\infty$.
\end{proof}
 
%
%

\begin{proof}[Proof of Corollary  \ref{ergodic}]  By Corollary \ref{spec}.C there is $\rho_0 \geq 0$ satisfying  (\ref{co}). Suppose that there is $S\in \mathbb{A}$ such that $S$ is $T$-invariant ($T^{-1}S=S$ up to subsets with zero  $\rho_0 m$ measure) and
\begin{eqnarray}\label{medida}   0< \int 1_S \ \rho_0 \ dm < 1.\end{eqnarray}
Then $\Phi(\rho_01_S)=\rho_01_S$ and $\Phi(\rho_01_{I\setminus S})=\rho_01_{I\setminus S}$. By  Theorem \ref{acim} and (\ref{medida}) we have that $\{ x\colon  \ \rho_0(x)1_S(x) > 0\}$ and  $\{ x\colon  \ \rho_0(x) 1_{I\setminus S}(x)> 0\}$  are disjoint subsets that are both non-empty countable unions of elements of $\cup_k \mathcal{P}^k$ (up to subsets of zero $m$-measure). Choose  $Q,P \in \cup_k \mathcal{P}^k$ such that 
$$Q\subset \{ x\colon  \ \rho_0(x) 1_S(x)> 0\} \ and \ P\subset \{ x\colon  \ \rho_0(x) 1_{I\setminus S}(x)> 0\}$$
By the transitivity  there is $n\geq 0$ such that $m(P\cap T^{-n}Q)> 0$. This contradicts the $T$-invariance of $S$. So $\rho_0$ is ergodic. This proves A. Now suppose that $\rho_0 m$ and $\rho_1 m$ are two distinct ergodic measures. Then $\{ x\colon  \ \rho_0(x)  > 0\}$ and  $\{ x\colon  \ \rho_1(x) > 0\}$  are disjoint and we can use an argument as above  and  arrive into a contradiction. This proves C.  Then B. follows from Corollary \ref{phy}. To prove C. note that by Corollary \ref{spec} we already know that $E= \sigma_{\mathcal{B}^s_{p,q}}(\Phi)\cap\mathbb{S}^1$ is a finite set such that $\lambda\in E$ implies $\lambda^{-1}\in E$. So if we prove that $\lambda_1,\lambda_2 \in E$ implies $\lambda_1\lambda_2 \in E$ then we show that $E$ is a cyclic group. Indeed, if $\Phi(u_i)=\lambda_i u_i$, $i=1,2$, $|u_i|_1=1$,  then we saw in the proof of Corollary \ref{spec}.D that $u_i=s_i(x)\rho_0(x)$, where $s_i(x)\in \mathbb{S}^1\cup \{0\}$ and for $m$-almost every $x$ we have $s_i(h_r(x))=\lambda_i s_i(x)$ provided $x\in J_r$, $g_r(x)\neq 0$ and $s_i(h_r(x))\neq 0$. Consequently $s_1(h_r(x))s_2(h_r(x))=\lambda_1 \lambda_2 s_1(x)s_2(x)$ under the same assumptions and one can easily see that $u=s_1s_2 \rho_0$ satisfies $\Phi(u)=\lambda_1\lambda_2  u$, so $\lambda_1\lambda_2\in E$. 
\end{proof}

%
%

\begin{proof}[Proof of Corollary  \ref{mixing}]  Let $\lambda \in E$ and $\Phi(u)=\lambda u$, with $|u|=1$. In the proof of Corollary  \ref{ergodic} we got  $u(x)=s(x)\rho_0(x)$, where $s(x)\in \mathbb{S}^1\cup \{0\}$ and for $m$-almost every $x$ we have $s(h_r(x))=\lambda s(x)$ provided $x\in J_r$,  $g_r(x)\neq 0$ and $s(h_r(x))\neq 0$. There is $k> n_0$ such that $\lambda^k=1$, so  $\Phi^k(u)= u$ and  $s(T^k(x))=s(x)$ for $\rho_0 m$-almost every $x$.  But $\Crr{A10}$ and Corollary \ref{ergodic} imply that $T^k$ is ergodic with respect to $\rho_0 m$, so $s$ is constant $\rho_0 m$-almost  everywhere and consequently $\lambda=1$.  So $E={1}$. The exponential decay of correlations follows from Corollary \ref{spec}.B.
\end{proof}

\section{Almost sure invariance principle.}

\begin{proof}[Proof of Corollary \ref{asip1}] By  Proposition \ref{besov-mult} in  \cite{smania-besov} we have that $e^{iv_1(x)  t}\in \mathcal{B}^{1/p}_{p,\infty}$ whenever $v_1 \in \mathcal{B}^{1/p}_{p,\infty}$ is real valued and  the pointwise  multiplier $M_t(w)=e^{iv_1(x)  t}w$ is a bounded operator in $B^s_{p,q}$ for every real $t$.
We have that  for $t$ small enough
\begin{align*}
|M_{t}f|_{\mathcal{B}^{s}_{p,q}}&= |(e^{iv_1    t}-1)f|_{\mathcal{B}^{s}_{p,q}} + |f|_{\mathcal{B}^{s}_{p,q}}  \\
&\leq (\frac{ |e^{iv_1    t}-1|_{\mathcal{B}^{1/p}_{p,\infty}}}{1-\Crr{G2}^{1/p-s}} +|e^{iv_1 t}-1|_\infty+1)   |f|_{\mathcal{B}^{s}_{p,q}} \\
&\leq  (\frac{C |t| |v_1|_{\mathcal{B}^{1/p}_{p,\infty}}}{1-\Crr{G2}^{1/p-s}} +2)   |f|_{\mathcal{B}^{s}_{p,q}}\\
&\leq  (\frac{C |t| |v_1  |_{\mathcal{B}^{1/p}_{p,\tilde{q}}}}{1-\Crr{G2}^{1/p-s}} +2)   |f|_{\mathcal{B}^{s}_{p,q}}\\
&\leq  3 |f|_{\mathcal{B}^{s}_{p,q}}.\end{align*} 
On the other hand if $v_2\in M(\mathcal{B}^s_{p.q})$ then there is $C$ such that 
$$|v_2f|_{\mathcal{B}^s_{p,q}}\leq C|f|_{\mathcal{B}^s_{p,q}}.$$
and using   the power series of $e^{iv_2t}$ one can easily see that $e^{iv_2t} \in M(\mathcal{B}^s_{p.q})$ and for $t$ small enough
$$|e^{itv_2} f|_{\mathcal{B}^s_{p,q}}\leq e^{Ct}|f|_{\mathcal{B}^s_{p,q}}\leq 2 |f|_{\mathcal{B}^s_{p,q}}.$$
Consequently for  $t$ small enough
$$|e^{itv}f|_{\mathcal{B}^{s}_{p,q}} \leq 6|f|_{\mathcal{B}^{s}_{p,q}}.$$
Consider the  operator $\Phi_t(f)=\Phi(e^{ivt}f)$. Note that 
$$|\Phi_t(f)|_1\leq |f|_1.$$
By  Theorem \ref{t1} we have that for $t$ small
$$|\Phi_t(f)|_{\mathcal{B}^s_{p,q}}\leq C |f|_1+6\Crr{GBS} \Crr{D}\Crr{DRSES}\Crr{GC}|f|_{\mathcal{B}^s_{p,q}}.$$
for some $C\geq 0$ and every $f\in \mathcal{B}^s_{p,q}$.
So if $6\Crr{GBS} \Crr{D}\Crr{DRSES}\Crr{GC} < 1$ we have 
$$|\Phi^n_t(f)|_{\mathcal{B}^s_{p,q}}\leq \frac{C}{1-6\Crr{GBS} \Crr{D}\Crr{DRSES}\Crr{GC}} |f|_1+(6\Crr{GBS} \Crr{D}\Crr{DRSES}\Crr{GC})^n|f|_{\mathcal{B}^s_{p,q}}.$$
for $t$ small enough and every $n$.  In particular
$$|\Phi^n_t|_{\mathcal{B}^s_{p,q}}\leq   \frac{CK_{t_0}}{1-6\Crr{GBS} \Crr{D}\Crr{DRSES}\Crr{GC}}+1  $$
for $t$ small enough and every $n$. Here $K_p$ is as in Proposition \ref{inc}.  Recall that by Proposition \ref{besov-lp} in  \cite{smania-besov}  we have $\rho_0 \in B^{s}_{p,q}\subset L^t$, where 
$$t = \frac{p}{1-sp} > 1,$$
and $v_1  \in B^{1/p}_{p,\infty} \subset L^u(m)$ for every $u \in [1,\infty)$, so since $v_2 \in M(B^{s}_{p,q})\subset L^\infty(m)$ we conclude that $v \in L^b(\rho_0 m)$ for every $b\in [1,\infty)$. 
By Corollary \ref{spec}.B we have that the sequence $v , v \circ T ,\dots$ satisfies the assumptions $(I)$ of Theorem 2.1 in Gou{\"e}zel \cite{gou1}. In particular the limit in (\ref{sigma}) indeed converges and if $\sigma > 0$ we have that (\ref{series}) satisfies the Central Limit Theorem with average $0$ and variance $\sigma$ and the Almost Sure Invariance Principle with every  error exponent 
$$\delta > \frac{b}{4b-4},$$
for every $b\in [1,\infty)$, so we can choose every $\delta > 1/4$.
\end{proof}

 \vspace{1cm}
\centerline{ \bf VI.  HOW TO DO IT.}
\addcontentsline{toc}{chapter}{\bf VI.  HOW TO DO IT.}
\vspace{1cm}

\section{Dynamical Slicing.}

\begin{proposition} \label{kkey} There is $\Crr{GSR}$, that depends only the good grid $\mathcal{P}$, with the following property. Let $\{I_r\}_{r\in \Lambda}$ be a countable family of pairwise disjoint  $(1-\beta p,\Crr{rp},t)$-strongly  regular domains in $(I,m,\mathcal{P})$ and $\alpha_r>  0$, for every $r\in \Lambda$. Let 
$$g=\sum_{k\geq t} \sum_{Q\in \mathcal{P}^k}  d_Qa_Q  $$
be a $\mathcal{B}^s_{p,q}$- representation, where $a_Q$ is the standard $(s,p)$-Souza's atom supported on $Q$. Assume that
$$T=\sup_{\substack{Q\in \mathcal{P}^k\\k\geq t}} \sum_{_{Q\cap I _r\neq \emptyset }} \alpha_r.$$
Then for every $r\in \Lambda$ there is a $\mathcal{B}^s_{p,q}$-representation
\begin{equation}\label{eqeee} g\cdot 1_{I_r}=\sum_k \sum_{\substack{ Q\in \mathcal{P}^k   \\Q\subset I_r}} c_Q^r a_Q,\end{equation} 
satisfying 
\begin{align}\label{eqeeee} \Big(  \sum_k  \big(\sum_r  \alpha_r^p \sum_{\substack{ Q\in \mathcal{P}^k   \\Q\subset I_r}}  |c_Q^r|^p  \big)^{q/p} \Big)^{1/q} \nonumber\\
\leq \Crr{GSR} T \Crr{rp}^{1/p}   \Big(  \sum_k \big( \sum_{Q\in \mathcal{P}^k}  |d_Q|^p \big)^{q/p} \Big)^{1/q}.  \end{align} 
\end{proposition} 
\begin{proof} If we apply  Proposition \ref{besov-pm1} in S. \cite{smania-besov} for the family of functions $\alpha_r 1_{I_r}$ we conclude that  there is  $\Crr{GSR}$, that depends only on the good grid $\mathcal{P}$, with the following property. There is a 
$\mathcal{B}^s_{p,q}$-representation
$$g\cdot \sum_r  \alpha_r  1_{I_r}=\sum_k \sum_{Q\in \mathcal{P}^k} c_Q a_Q$$
satisfying 
$$\Big(  \sum_k  \big(\sum_{Q\in \mathcal{P}^k}  |c_Q|^p  \big)^{q/p} \Big)^{1/q}\leq \Crr{GSR} T\Crr{rp}^{1/p}   \Big(  \sum_k \big( \sum_{Q\in \mathcal{P}^k}  |d_Q|^p \big)^{q/p} \Big)^{1/q}  $$
and if $c_Q\neq 0$ then $Q\subset I_r$, for some $r\in \Lambda$. Such $r$ in our case must be unique, since the sets in the family $\{I_r\}$ are pairwise disjoint.  So if $Q\subset I_r$ define $c_Q^r=c_Q/\alpha_r.$ 
 It is easy to see that (\ref{eqeee}) and (\ref{eqeeee}) hold. 
\end{proof} 

\begin{proof}[Proof of Theorem \ref{pa}]   Let 
$$f=\sum_{k} \sum_{Q\in \mathcal{P}^k}  d_Qa_Q  $$
be a $\mathcal{B}^s_{p,q}$- representation, where $a_Q$ is the standard $(s,p)$-Souza's atom supported on $Q$.  Consider also the  $\mathcal{B}^s_{p,q}$-representations 
$$f_1= \sum_{k\geq   t} \sum_{Q\in \mathcal{P}^k}  d_Qa_Q.$$
$$f_2= \sum_{k< t} \sum_{Q\in \mathcal{P}^k}  d_Qa_Q.$$

\noindent {\bf Step I.}   We can apply  Proposition \ref{kkey}  to the family  $\{ I_r\}_{r\in \Lambda}$, taking $g=f_1$, $\alpha_r=1$  and $T=M$. So for each $r\in \Lambda$  there  is a $\mathcal{B}^s_{p,q}$-representation 
$$f_1 \cdot 1_{I_r}=\sum_k \sum_{\substack{ Q\in \mathcal{P}^k   \\Q\subset I_r}} c^r_Q a_Q$$
satisfying 
$$\Big(  \sum_k  \big(\sum_r \sum_{\substack{ Q\in \mathcal{P}^k   \\Q\subset I_r}}  |c_Q^r|^p  \big)^{q/p} \Big)^{1/q}\leq \Crr{GSR} M \Crr{rp1f}^{1/p}   \Big(  \sum_k \big( \sum_{Q\in \mathcal{P}^k}  |d_Q|^p \big)^{q/p} \Big)^{1/q}. $$

Note that  
\begin{align*}&\Big( \sum_j\big( \sum_{r\in \Lambda} \Theta_r (\sum_{\substack{ Q\in \mathcal{P}^j \\  Q \subset I_r }}  |c_Q^r|^p )^{1/p}  \big)^{q}\Big)^{1/q}\\
&\leq   \Big( \sum_{r \in \Lambda}  \Theta_r^{p'}   \Big)^{1/p'}  \Big( \sum_j\big(    \sum_{r \in \Lambda}\sum_{\substack{ Q\in \mathcal{P}^j \\  Q \subset I_r }}  |c_Q^r|^p \big)^{q/p}  \Big)^{1/q}\\
&\leq    \Big( \sum_{r \in \Lambda} \Theta_r^{p'}   \Big)^{1/p'}    \Crr{GSR} M \Crr{rp1f}^{1/p}   \Big(  \sum_k \big( \sum_{Q\in \mathcal{P}^k}  |d_Q|^p \big)^{q/p} \Big)^{1/q}.
\end{align*}
So
$$(\mathcal{I},  \sum_{k\geq t} \sum_{Q\in \mathcal{P}^k}  d_Qa_Q )$$
has a $\Crr{GSR}   \Crr{f1f}$-slicing.\\

\noindent {\bf Step II.} Note that since $I_r$, $r\in \Lambda$,  is a $(1-\beta p,\Crr{rp1f},t )$-strongly  regular  domain then  $I_r$ is also a $(1-\beta p,\Cll{rp44},0)$-strongly  regular domain, where 
$$\Crr{rp44}=\Crr{rp1f}  \Crr{G1}^{-t}$$ depends only on the grid $\mathcal{P}$, $\Crr{rp1f}$ and $t$.   Now   apply  Proposition \ref{kkey}  to the family  $\{ I_r\}_{r\in \Lambda}$, with $g=f_2$ and taking $\alpha_r=1$ and  $T=\#\Lambda$. We conclude that  for every $i\in \Lambda$ there exists a  $\mathcal{B}^s_{p,q}$-representation
$$f_2 \cdot 1_{I_r}=\sum_k \sum_{\substack{ Q\in \mathcal{P}^k   \\Q\subset I_r}} c^r_Q a_Q$$
satisfying 
$$\Big(  \sum_k  \big(\sum_r \sum_{\substack{ Q\in \mathcal{P}^k   \\Q\subset I_r}}  |c_Q^r|^p  \big)^{q/p} \Big)^{1/q}\leq (\#\Lambda)  \Crr{GSR}  \Crr{rp44}^{1/p}   \Big(  \sum_k \big( \sum_{Q\in \mathcal{P}^k}  |d_Q|^p \big)^{q/p} \Big)^{1/q}. $$
The same argument as in Step I gives
\begin{align*}&\Big( \sum_j\big( \sum_{r\in \Lambda} \Theta_r  (\sum_{\substack{ Q\in \mathcal{P}^j \\  Q \subset I_r }}  |c_Q^r|^p )^{1/p}  \big)^{q}\Big)^{1/q}\\
&\leq    \Big( \sum_{r \in \Lambda}  \Theta_r^{p'}  \Big)^{1/p'}  (\#\Lambda)  \Crr{GSR}  \Crr{rp44}^{1/p}      \Big(  \sum_k \big( \sum_{Q\in \mathcal{P}^k}  |d_Q|^p \big)^{q/p} \Big)^{1/q},
\end{align*}
so we conclude  that 
$$(\mathcal{I},  \sum_{k< t} \sum_{Q\in \mathcal{P}^k}  d_Qa_Q )$$
has a $\Crr{GSR}   \Crr{f2f}$-slicing.
\end{proof}

\begin{proof}[Proof of Theorem \ref{pa2}]   Let 
$$f=\sum_{k} \sum_{Q\in \mathcal{P}^k}  d_Qa_Q  $$
be a $\mathcal{B}^s_{p,q}$- representation, where $a_Q$ is the standard $(s,p)$-Souza's atom supported on $Q$.  Consider also the  $\mathcal{B}^s_{p,q}$-representations 
$$f_1= \sum_{k\geq   t} \sum_{Q\in \mathcal{P}^k}  d_Qa_Q.$$
$$f_2= \sum_{k< t} \sum_{Q\in \mathcal{P}^k}  d_Qa_Q.$$

\noindent {\bf Step I.}   We can apply  Proposition \ref{kkey}  to the family  $\{ I_r\}_{r\in \Lambda_1}$, taking $g=f_1$, $\alpha_r=\Theta_r$.  So for each $r\in \Lambda$  there  is a $\mathcal{B}^s_{p,q}$-representation 
$$f_1 \cdot 1_{I_r}=\sum_k \sum_{\substack{ Q\in \mathcal{P}^k   \\Q\subset I_r}} c^r_Q a_Q$$
satisfying 
\begin{align*}&N^{1/p'} \Big( \sum_j\big( \sum_{r\in \Lambda} \Theta_r^p \sum_{\substack{ Q\in \mathcal{P}^j \\  Q \subset I_r }}  |c_Q^r|^p   \big)^{q/p}\Big)^{1/q}\\
&\leq   N^{1/p'} T  \Crr{GSR}  \Crr{rp1f}^{1/p}   \Big(  \sum_k \big( \sum_{Q\in \mathcal{P}^k}  |d_Q|^p \big)^{q/p} \Big)^{1/q}.
\end{align*}
So
$$(\mathcal{I},  \sum_{k\geq t} \sum_{Q\in \mathcal{P}^k}  d_Qa_Q )$$
has a $\Crr{GSR}   \Crr{f1a}$-slicing. \\

\noindent {\bf Step II.} As in Step II. of the proof of Proposition \ref{pa} we have that  $I_r$, $r\in \Lambda$,  is a $(1-\beta p,\Cll{rp4},0)$-strongly  regular domain, where 
$$\Crr{rp4}=\Crr{rp1f}  \Crr{G1}^{-t}$$ depends only on the grid $\mathcal{P}$, $\Crr{rp1f}$ and $t$.   Now   apply  Proposition \ref{kkey}  to the family  $\{ I_r\}_{r\in \Lambda}$, with $g=f_2$ and taking $\alpha_r=\Theta_r$. We conclude that  for every $i\in \Lambda_1$ there exists a  $\mathcal{B}^s_{p,q}$-representation
$$f_2 \cdot 1_{I_r}=\sum_k \sum_{\substack{ Q\in \mathcal{P}^k   \\Q\subset I_r}} c^r_Q a_Q$$
satisfying 
\begin{align*}
&N^{1/p'} \Big(  \sum_k  \big(\sum_r \Theta_r  \sum_{\substack{ Q\in \mathcal{P}^k   \\Q\subset I_r}}  |c_Q^r|^p  \big)^{q/p} \Big)^{1/q}\\
&\leq N^{1/p'} (\#\Lambda) ( \sup_{r\in \Lambda} \Theta_r) \Crr{GSR}  \Crr{rp4}^{1/p}   \Big(  \sum_k \big( \sum_{Q\in \mathcal{P}^k}  |d_Q|^p \big)^{q/p} \Big)^{1/q},
\end{align*}
so we obtain that 
$$(\mathcal{I},  \sum_{k< t} \sum_{Q\in \mathcal{P}^k}  d_Qa_Q )$$
has a $\Crr{GSR}   \Crr{f2a}$-slicing.
\end{proof}

\begin{proof}[Proof of Theorem \ref{tt1}]  For every $r\in \Lambda$ we apply (as usual) Proposition \ref{kkey}, this time for the family with a unique element $\{ I_r\}$, with $g=f$ and taking $\alpha_r=1$ and  $T=1$. We conclude that  for every $r\in \Lambda$ there exists a  $\mathcal{B}^s_{p,q}$-representation
$$f \cdot 1_{I_r}=\sum_k \sum_{\substack{ Q\in \mathcal{P}^k   \\Q\subset I_r}} c^r_Q a_Q$$
satisfying 
$$\Big(  \sum_k  \big(\sum_{\substack{ Q\in \mathcal{P}^k   \\Q\subset I_r}}  |c_Q^r|^p  \big)^{q/p} \Big)^{1/q}\leq  \Crr{GSR}  \Crr{rp36}^{1/p}   \Big(  \sum_k \big( \sum_{Q\in \mathcal{P}^k}  |d_Q|^p \big)^{q/p} \Big)^{1/q}. $$
Of course
\begin{align*} &\Big(  \sum_k  \big(  \sum_r   \Theta_r \big(\sum_{\substack{ Q\in \mathcal{P}^k   \\Q\subset I_r}}  |c_Q^r|^p \big)^{1/p}  \big)^{q} \Big)^{1/q}\\
&\leq \sum_r \Theta_r \Big(  \sum_k  \big(\sum_{\substack{ Q\in \mathcal{P}^k   \\Q\subset I_r}}  |c_Q^r|^p \big)^{q/p} \Big)^{1/q}\\
&\leq  \Crr{GSR}  \Crr{rp36}^{1/p}   \big( \sum_r \Theta_r\big) \Big(  \sum_k \big( \sum_{Q\in \mathcal{P}^k}  |d_Q|^p \big)^{q/p} \Big)^{1/q}. \end{align*}
\end{proof}

\begin{proof}[Proof of Theorem \ref{tt2}]  We apply again  Proposition \ref{kkey}, this time for the family with a unique element $\{ \Omega\}$, with $\Omega=\cup_r I_r$,  $g=f$ and taking  $T=1$. We conclude that   there exists a  $\mathcal{B}^s_{p,q}$-representation
$$f \cdot 1_{\Omega}=\sum_k \sum_{\substack{ Q\in \mathcal{P}^k}}c_Q a_Q$$
satisfying 
$$\Big(  \sum_k  \big(\sum_{\substack{ Q\in \mathcal{P}^k}}  |c_Q|^p  \big)^{q/p} \Big)^{1/q}\leq  \Crr{GSR}  \Crr{rp36}^{1/p}   \Big(  \sum_k \big( \sum_{Q\in \mathcal{P}^k}  |d_Q|^p \big)^{q/p} \Big)^{1/q}. $$
 if $c_Q\neq 0$ then $Q\subset \Omega$, so by assumption $Q\subset I_r$, for some $r\in \Lambda$. Such $r$  must be unique, since the sets in the family $\{I_r\}$ are pairwise disjoint.  So if $Q\subset I_r$ define $c_Q^r=c_Q.$  Then
 $$f \cdot 1_{I_r}=\sum_k \sum_{\substack{ Q\in \mathcal{P}^k   \\Q\subset I_r}} c^r_Q a_Q$$
 for every $r$ and
\begin{align*} &N^{1/p'} \Big(  \sum_k  \big(   \sum_r   \Theta_r^p  \sum_{\substack{ Q\in \mathcal{P}^k   \\Q\subset I_r}}  |c_Q^r|^p \big)^{q/p}\Big)^{1/q}\\
&\leq N^{1/p'} \big( \sup_r \Theta_r \big) \Big(  \sum_k  \big( \sum_r   \sum_{\substack{ Q\in \mathcal{P}^k   \\Q\subset I_r}}  |c_Q^r|^p \big)^{q/p} \Big)^{1/q}\\
&\leq  N^{1/p'} \big( \sup_r \Theta_r \big) \Crr{GSR}  \Crr{rp36}^{1/p}   \big( \sum_r \Theta_r\big) \Big(  \sum_k \big( \sum_{Q\in \mathcal{P}^k}  |d_Q|^p \big)^{q/p} \Big)^{1/q}. \end{align*}
\end{proof}

\section{Boundness on Lebesgue spaces} We have that  $\Phi$, under very mild conditions,  defines a  bounded transformation from  $L^{t_0}(m)$ to $L^{1}(\mu)$.
 \label{boundlp}

\begin{corollary} [Boundeness on Lebesgue spaces] \label{b22} Let $\epsilon'$ be such that $\delta=\epsilon -\epsilon'$ and $\epsilon, \epsilon'> 0$.
Suppose that 
$$\Lambda=\Lambda_1\cup\Lambda_2\cup\Lambda_3,$$
with $\Lambda_i\cap \Lambda_k=\emptyset$, for $i\neq k$,  such that
\begin{itemize}
\item[i.] If $i,j \in \Lambda_1$, with $i\neq j$, then  $I_i\cap I_j =\emptyset$. Moreover  there is $\Cll{113}\geq 0$ such that for every $i\in \Lambda_1$  and $Q\in \mathcal{P}$ such that $Q \subset I_i$ 
\begin{equation}\label{supe52} \sup(|g|,h_i^{-1} Q) \leq  \Crr{113} \frac{|Q|}{|h_i^{-1}Q|}\end{equation} 
\item[ii.] We have \begin{equation}\label{supe22} \sup(|g|,h_r^{-1} Q) \leq  \Cll{11} \Big( \frac{|Q|}{|h_r^{-1}Q|}\Big)^{1/p-s+\epsilon}\end{equation} 
for every $Q\in \mathcal{P}$ such that $Q \subset I_r$, with $r\in \Lambda_2\cup \Lambda_3$, 
\item[iii.] We have 
\begin{equation} \sum_{r\in \Lambda_2}    \Crr{DC2}^{|\epsilon'|  a_r} < \infty.\end{equation}
\item[iv.] If $i,j \in \Lambda_3$, with $i\neq j$, then  $I_i\cap I_j =\emptyset$. Moreover
$$\sum_{r\in \Lambda_3} \Crr{DC2}^{ t_0' a_r |\epsilon'|} <\infty,$$
where $t_0'$ satisfies $1/t_0'+ 1/t_0=1$. 
\end{itemize}
For every $f\in L^{1}$ and $r\in \Lambda$ consider  the measurable functions $\Phi_r(f)\colon I \rightarrow \mathbb{C}$ defined by 
$$\Phi_r(f)= g_r(x)f(h_r(x))$$
if $x \in J_r$ and $\Phi_r(f)(x)=0$ otherwise. Then

\begin{itemize}
\item[A.] For every $r\in \Lambda_1$ and $f\in L^{1}(m)$ we have that  $\Phi_r(f)$ belongs to $L^{1}(\mu)$,  
$$\Phi_r\colon L^{1}(m)\rightarrow L^{1}(\mu)$$
is a bounded linear transformation and 
$$ |\Phi_r(f)|_{1} \leq \Crr{113}  |f 1_{I_r}|_{1}.$$
In particular 
$$ \sum_{r\in \Lambda_1} |\Phi_r(f)|_{1} \leq \Crr{113} |f|_1$$
\item[B.] For every $r\in \Lambda_2\cup\Lambda_3$ and $f\in L^{t_0}(m)$ we have that  $\Phi_r(f)$ belongs to $L^{t_0}(\mu)$,  and 
$$\Phi_r\colon L^{t_0}(m)\rightarrow L^{t_0}(\mu)$$
is a bounded linear transformation and 
$$ |\Phi_r(f)|_{t_0} \leq  \Crr{11}  \Crr{DC1}^{|\epsilon'|}\Crr{DC2}^{ a_r |\epsilon'|} |f 1_{I_r}|_{t_0}.$$
In particular
$$ \sum_{r\in \Lambda_2} |\Phi_r(f)|_{1} \leq  \Crr{11} \Crr{DC1}^{|\epsilon'|} \big( \sum_{r\in \Lambda_2}\Crr{DC2}^{ a_r |\epsilon'|} \big) |f|_{t_0}$$
and
$$ \sum_{r\in \Lambda_3} |\Phi_r(f)|_{1} \leq \Crr{11} \Crr{DC1}^{|\epsilon'|} \big( \sum_{r\in \Lambda_3} \Crr{DC2}^{ t_0' a_r |\epsilon'|}\big)^{1/t_0'} |f|_{t_0}.$$
\item[C.] In particular  $\Phi f $ is well defined and belongs to $L^1(\mu)$ for every $f\in L^{t_0}(m)$ and
$$\Phi\colon L^{t_0}(m)\rightarrow L^1(\mu)$$
is a continuous linear transformation, with 
$$|\Phi f|_{1}\leq \sum_r |\Phi_r(f)|_{1}\leq \Crr{999}|f|_{t_0}$$
for some $\Cll{999}\geq0$.
\end{itemize}
\end{corollary}

\begin{proof} Recall
$$\mathcal{A}= \{Q \in \mathcal{P}, \  Q\subset I_r, \text{ for some $r\in \Lambda$}\}\cup \{Q \in \mathcal{P}, \  Q\cap I_r=\emptyset, \text{ for every $r\in \Lambda$}\}.$$
If $f$ is a (finite)  linear combinations of  characteristic functions of sets in $\mathcal{A}$  we can write 
$$f1_{I_r} =\sum_{i\leq n} c_i 1_{Q_i},$$
with $Q_i \subset I_r$. We can assume that $\{Q_i\}_i$ is a family of pairwise disjoint sets. In particular
$$|f1_{I_r}|_{t_0}^{t_0} =\sum_{i\leq n} |c_i|^{t_0}  |Q_i|,$$
and
$$\Phi_r(f)=  \sum_{i\leq n} c_i g 1_{h_r^{-1}Q_i},$$
Then for every $r\in \Lambda_2\cup\Lambda_3$
\begin{align*} 
|\sum_{i\leq n} c_i g 1_{h_r^{-1}Q_i}|_{t_0}^{t_0} &= \int   \Big|\sum_{i\leq n}  c_i  g 1_{h_r^{-1}Q_i}\Big|^{t_0} \ dm \\
 &= \int  \sum_{i\leq n} |c_i|^{t_0}   |g|^{t_0} 1_{h_r^{-1}Q_i} \ dm \\
&\leq  \Crr{11}^{t_0} \sum_{i\leq n}  |c_i|^{t_0} \Big( \frac{|Q_i|}{|h^{-1}_rQ_i|}\Big)^{(1/p-s+\epsilon)t_0} |h_r^{-1}Q_i| \\
&\leq  \Crr{11}^{t_0} \sum_{i\leq n}  |c_i|^{t_0} \Big( \frac{|Q_i|}{|h^{-1}_rQ_i|}\Big)^{ 1+t_0\epsilon'} |h_r^{-1}Q_i| \\
&\leq  \Crr{11}^{t_0} \sum_{i\leq n}  |c_i|^{t_0} \Big( \frac{|Q_i|}{|h^{-1}_rQ_i|}\Big)^{t_0\epsilon'} |Q_i| \\
&\leq \Crr{11}^{t_0}  \Crr{DC1}^{|\epsilon'|t_0}  \Crr{DC2}^{ a_r |\epsilon'| t_0}  \sum_{i\leq n}  |c_i|^{t_0}  |Q_i| \\
&\leq \Crr{11}^{t_0}  \Crr{DC1}^{|\epsilon'|t_0}   \Crr{DC2}^{ a_r |\epsilon' |t_0}  |f1_{I_r}|_{t_0}^{t_0}.
\end{align*}
Since linear combinations of  characteristic functions of sets in $\mathcal{A}$  are dense in $L^{t_0}$. we conclude that  for {\it every} $f\in L^{t_0}$ and $r\in \Lambda_2\cup\Lambda_3$ we have
$$|\Phi_r(f)|_{t_0}\leq   \Crr{11}   \Crr{DC1}^{|\epsilon'|} \Crr{DC2}^{ a_r |\epsilon'| }  |f1_{I_r}|_{t_0}.$$
So
\begin{eqnarray*}
\sum_{r\in \Lambda_2} |\Phi_r(f)|_{1}\leq \sum_{r\in \Lambda_2} |\Phi_r(f)|_{t_0}&\leq& \Crr{11}  \Crr{DC1}^{|\epsilon'|} \sum_{r\in \Lambda_2} \Crr{DC2}^{ a_r|\epsilon'| }  |f 1_{I_r}|_{t_0}\\
&\leq& \Crr{11}  \Crr{DC1}^{|\epsilon'|}  \big( \sum_{r\in \Lambda_2} \Crr{DC2}^{ a_r |\epsilon'| } \big) |f|_{t_0},
\end{eqnarray*}
and 
\begin{eqnarray*}
\sum_{r\in \Lambda_3} |\Phi_r(f)|_{1}\leq  \sum_{r\in \Lambda_3} |\Phi_r(f)|_{t_0}&\leq& \Crr{11} \Crr{DC1}^{|\epsilon'|}  \sum_{r\in \Lambda_3} \Crr{DC2}^{ a_r |\epsilon'| }  |f 1_{I_r}|_{t_0}\\
&\leq& \Crr{11}  \Crr{DC1}^{|\epsilon'|} \big( \sum_{r\in \Lambda_3} \Crr{DC2}^{ t_0' a_r |\epsilon'|}\big)^{1/t_0'}  \big( \sum_{r\in \Lambda_3} |f 1_{I_r}|_{t_0}^{t_0}\big)^{1/t_0}\\
&\leq& \Crr{11} \Crr{DC1}^{|\epsilon'|} \big( \sum_{r\in \Lambda_3} \Crr{DC2}^{ t_0' a_r |\epsilon'|}\big)^{1/t_0'}  \big(  \sum_{r\in \Lambda_3}  |f1_{I_r}|_{t_0}^{t_0}\big)^{1/t_0}\\
&\leq& \Crr{11} \Crr{DC1}^{|\epsilon'|} \big( \sum_{r\in \Lambda_3} \Crr{DC2}^{ t_0' a_r |\epsilon'|}\big)^{1/t_0'}  |f|_{t_0}.
\end{eqnarray*}
If   $r\in \Lambda_1$ we have
\begin{align*} 
|\sum_{i\leq n} c_i g 1_{h_r^{-1}Q_i}|_{1} &= \int  \Big| \sum_{i\leq n}  c_i  g 1_{h_r^{-1}Q_i} \Big| \ dm \\
 &= \int  \sum_{i\leq n} |c_i|   |g| 1_{h_r^{-1}Q_i}  \ dm \\
&\leq  \Crr{113} \sum_{i\leq n}  |c_i| \frac{|Q_i|}{|h^{-1}_rQ_i|}   |h_r^{-1}Q_i| \\
&\leq \Crr{113}    \sum_{i\leq n}  |c_i|  |Q_i| \\
&\leq \Crr{113}   |f 1_r|_1.
\end{align*}
Since linear combinations $f$ of  characteristic functions of sets in $\mathcal{G}$  are dense in $L^{1}$. we conclude that  for every $f\in L^{1}$ and $r\in \Lambda_1$
$$|\Phi_r(f)|_1\leq \Crr{113}   |f 1_r|_1,$$
so
$$\sum_{r\in \Lambda_1} |\Phi_r(f)|_{1}\leq  \Crr{113} \sum_{r\in \Lambda_1}  |f 1_r|_1\leq  \Crr{113}  |f|_1\leq \Crr{113}  |f|_{t_0}.$$
\end{proof}

\bibliographystyle{abbrv}
\bibliography{bibliografia}

\end{document}